\documentclass[reqno, 11pt]{article}

\usepackage{enumerate}
\usepackage{needspace}
\usepackage{latexsym}
\usepackage{amsmath}
\usepackage{amsfonts}
\usepackage{amsthm}
\usepackage{amssymb,mathtools}
\usepackage{newlfont}
\usepackage{graphics}
\usepackage{color}
\usepackage{float}
\usepackage{diagbox}
\usepackage{tocloft}
\usepackage{titlesec}
\usepackage{booktabs,array,longtable,enumitem,microtype}
\usepackage{extpfeil}
\usepackage{centernot}

\usepackage[linesnumbered,ruled,vlined]{algorithm2e}
\usepackage{url}

\usepackage[margin=1in]{geometry}

\usepackage[T1]{fontenc}
\usepackage{lmodern}
\usepackage{microtype}
\usepackage{rotating}
\usepackage{multirow}
\usepackage{extarrows}
\usepackage[sort,compress,numbers]{natbib}
\usepackage[utf8]{inputenc}
\usepackage{xcolor}
\usepackage{listings}
\usepackage{aliascnt}
\usepackage[pagebackref,colorlinks=true,linkcolor=blue,citecolor=red,urlcolor=blue]{hyperref}
\usepackage[nameinlink,noabbrev,capitalize]{cleveref}
\numberwithin{equation}{section}

\setlist[enumerate]{label=\textup{(\roman*)},leftmargin=*}
\newtheorem{theorem}{Theorem}[section]
\newaliascnt{proposition}{theorem}
\newtheorem{proposition}[proposition]{Proposition}
\aliascntresetthe{proposition}
\newaliascnt{lemma}{theorem}
\newtheorem{lemma}[lemma]{Lemma}
\aliascntresetthe{lemma}
\newaliascnt{corollary}{theorem}
\newtheorem{corollary}[corollary]{Corollary}
\aliascntresetthe{corollary}

\theoremstyle{definition}
\newaliascnt{definition}{theorem}
\newtheorem{definition}[definition]{Definition}
\aliascntresetthe{definition}
\newaliascnt{remark}{theorem}

\aliascntresetthe{remark}
\newaliascnt{example}{theorem}

\aliascntresetthe{example}
\newaliascnt{conjecture}{theorem}
\newtheorem{conjecture}[conjecture]{Conjecture}
\aliascntresetthe{conjecture}
\newaliascnt{question}{theorem}

\aliascntresetthe{question}
\newaliascnt{problem}{theorem}

\aliascntresetthe{problem}
\newaliascnt{alg}{theorem}

\aliascntresetthe{alg}

\allowdisplaybreaks[4]

\SetKwInput{KwInput}{Input}                
\SetKwInput{KwOutput}{Output}              

\DeclareMathOperator{\CT}{CT}
\DeclareMathOperator{\Res}{Res}

\newcommand{\NN}{\mathbb Z_{\geq0}}
\newcommand{\ZZ}{\mathbb Z}
\newcommand{\RR}{\mathbb R}
\newcommand{\CC}{\mathbb C}

\newcommand{\dd}{\mathrm d}

\title{Real-Rootedness and Gamma-Positivity for a Variation of the Morris Constant Term}

\author{Feihu Liu$^{\color{blue} \dag}$ and Zihao Zhang$^{\color{blue} \S}$
\\[2mm]
{\small $^{\color{blue} \dag}$ Center for Combinatorics, LPMC,}\\[-0.8ex]
{\small Nankai University, Tianjin 300071, P.R.~China}\\
{\small $^{\color{blue} \S}$ School of Mathematics and Statistics,}\\[-0.8ex]
{\small Beijing Institute of Technology, Beijing 102400, P.R.~China}\\
{\small {\color{blue} $^\dag$} Email address: liufeihu7476@163.com}\\
{\small {\color{blue} $^\S$} Email address: zihao-zhang@foxmail.com}\\
}

\date{\today}

\begin{document}

\maketitle

\begin{abstract}
Beck and Pixton expressed the Ehrhart polynomial of the Birkhoff polytope as a weighted sum of constant terms of several multivariate rational functions. Xin and Zhang studied a class of constant terms $h_n(t)$, which can be regarded as a variation of the Morris constant term. They proved that $h_n(t)$ is a polynomial of degree $(n-1)^2$ and obtained many nice properties involving the Morris constant term identity. Let $h_n^*(y)=(1-y)^{(n-1)^2+1}\sum_{t\geq0}h_n(t)y^t$. For fixed $n\geq 3$, we obtain the following results for $h_n^*(y)$: 
(i): $h_n^*(y)$ is a polynomial with positive integer coefficients.
(ii): $h_n^*(y)$ is real-rooted. In particular, all its roots are non-positive real numbers.
(iii): $h_n^*(y)$ is Gamma-positive.
Furthermore, $h_n^*(y)$ is palindromic, unimodal, and ultra log-concave.
This confirms Xin and Zhang's conjecture regarding $h_n^*(y)$. 
In order to resolve this conjecture, we also developed an operator on the space of symmetric polynomials that preserves real stability. As a byproduct, we prove that every root of a Gamma polynomial associated with $h_n^*(y)$ is a negative real number.
\end{abstract}

\noindent
\begin{small}
\end{small}

\noindent
\begin{small}
\end{small}

\tableofcontents

\section{Introduction}

\subsection{Background and preliminaries}

Let $n\geq3$.  The Birkhoff polytope of order $n$ is
\[
B_n=\left\{(b_{ij})\in\RR_{\geq0}^{n\times n}:
 \sum_{j=1}^n b_{ij}=1\ (1\leq i\leq n),\quad
 \sum_{i=1}^n b_{ij}=1\ (1\leq j\leq n)\right\}.
\]
Its Ehrhart counting function,
$H_n(t)=\#(tB_n\cap\ZZ^{n\times n})$ for $t\in\NN$,
counts nonnegative integral matrices with every row and column sum
equal to $t$.  Beck and Pixton~\cite{BeckPixton} express this
function in terms of constant terms.  
Inspired by results concerning the structure of the constant term~\cite{xin2022},  Xin and Zhang \cite[Section~1]{XZ} derived a particular summand for $H_n(t)$, with a sign normalization, denoted by $h_n(t)$.

To specify this summand, let
$\boldsymbol r=(r_1,\ldots,r_n)$ range over the weak compositions
of $n$, meaning the vectors of nonnegative integers with
$r_1+\cdots+r_n=n$, and set
\[
 C_{\boldsymbol r}(t)=\CT_{x_1,\ldots,x_n}
 \prod_{i=1}^n
 \frac{x_i^{(r_i-1)t}}{\prod_{j\ne i}(1-x_j/x_i)^{r_i}}.
\]
Each rational function here is expanded in the field $\mathbb Q((x_1))\cdots((x_n))$. 
The constant term operator $\CT_{x_1,\ldots,x_n} F$ denotes the constant term of Laurent series $F$ with respect to the variables $(x_1,\ldots,x_n)$.
In the notation of \cite[Section~1 and (1.2)]{XZ}, the decomposition and the chosen component are
\begin{equation}\label{eq:birkhoff-component}
\begin{split}
 H_n(t)&=\sum_{\boldsymbol r}\frac{n!}{r_1!\cdots r_n!}\,C_{\boldsymbol r}(t),\\
 h_n(t)&=(-1)^{(n-1)(n-2)/2}C_{(2,1,\ldots,1,0)}(t).
\end{split}
\end{equation}
We study the numerator associated with $h_n(t)$.  It is distinct
from the $h^*$-polynomial of $B_n$.

We use the following explicit definition from \cite[(1.2)--(1.3)]{XZ}:
\begin{equation}\label{eq:original}
h_n(t)=\CT_{x_2,\ldots,x_n}\frac{(1-x_n)\displaystyle\prod_{i=2}^{n-1}x_i^{n-1}
\prod_{i=2}^{n-1}(x_i-x_n)}{x_n^t\displaystyle\prod_{i=2}^{n}(1-x_i)^3
\prod_{2\leq i<j\leq n}(x_i-x_j)^2},
\qquad t\in\NN.
\end{equation}
Here the rational function is expanded in $\mathbb Q((x_2))\cdots((x_n))$.  The operator extracts the
coefficient of $x_2^0\cdots x_n^0$ in that field.  The variable order specifies the formal expansion; see \cref{sec:ct}.

The constant term identity \eqref{eq:original} can also be regarded as a variant of the Morris constant term identity.
The following Morris constant term identity was first proved in \cite{MorrisBV}, and a simple proof was subsequently given in \cite{XinThesis}:
\begin{align}\label{Equation-Morrio-Equa}
\CT_{x_1, \ldots, x_n} M(n; k_1, k_2, k_3) = \prod_{j=0}^{n-1} \frac{\Gamma(1 + \frac{k_3}{2}) \Gamma(k_1 + k_2 - 1 + (n + j - 1)\frac{k_3}{2})}{\Gamma(1 + (j + 1)\frac{k_3}{2}) \Gamma(k_1 + j\frac{k_3}{2}) \Gamma(k_2 + j\frac{k_3}{2})},
\end{align}
where
\begin{equation*}
M(n; k_1, k_2, k_3) := \frac{1}{\prod_{i=1}^n x_i^{k_1-1} \prod_{i=1}^n (1 - x_i)^{k_2} \prod_{1 \le i < j \le n} (x_i - x_j)^{k_3}},\quad k_i \in \mathbb{N},\quad k_1 + k_2 \ge 2,
\end{equation*}
and \(M(0; k_1, k_2, k_3) = 1\).

Morris introduced his original constant term identity in his Ph.D. thesis \cite{MorrisThesis}. 
This corresponds to the case where $k_3$ is a negative even integer.
This identity is equivalent to the Selberg integral; see \cite{selberg1944}.
Zeilberger \cite{Zeilberger} was the first to consider the case \((k_1, k_2, k_3) = (1, 2, 1)\).
Equation \eqref{Equation-Morrio-Equa} was subsequently studied in \cite{morales2021}.
For an excellent overview regarding the Selberg integral, we refer the reader to \cite{forrester2008}.  
For more on Morris constant term identity, see \cite{ZhouYue,ZhouLuFu}.

Based on Equation \eqref{eq:original}, we define the following series, initially in $\mathbb Q[[y]]$, 
\begin{equation}\label{eq:hstar-definition}
h_n^*(y)=(1-y)^{(n-1)^2+1}\sum_{t\geq0}h_n(t)y^t.
\end{equation}
Xin and Zhang~\cite[Conjecture~5.1]{XZ} predicted that $h_n^*(y)$ is a polynomial with positive integer coefficients, and they proposed the following conjecture.

\begin{conjecture}(Xin--Zhang, \cite[Conjecture~5.1]{XZ})\label{conj:xz}
For every integer $n\geq3$, the series $h_{n}^*(y)$ is of the form
\begin{equation*}
h_{n}^*(y):=(1-y)^{(n-1)^2+1}\sum_{t\geq 0}h_n(t)y^t =y^{\left\lfloor \frac{(n-1)^2}{4}\right\rfloor}\sum_{i=0}^{\left\lfloor \frac{(n-2)^2}{2}\right\rfloor} a_{n,i} y^i,
\end{equation*}
where all $a_{n,i}$ are positive integers.  Moreover, $h_{n}^*(y)$ is real-rooted, and the reduced coefficient sequence $\left(a_{n,0},a_{n,1},\ldots, a_{n,\lfloor \frac{(n-2)^2}{2}\rfloor}\right)$
is palindromic, unimodal, and ultra-log-concave in degree $\left\lfloor \frac{(n-2)^2}{2}\right\rfloor$.
\end{conjecture}

Xin and Zhang have verified \cref{conj:xz} for $3\leq n\leq 29$. The data of $h^*(y)$ is also available at~\cite{XZ}.
We prove the conjecture together with gamma positivity.  In particular, polynomiality follows from the proof.
In what follows, we introduce the various terminology mentioned above.

For a real sequence $(a_0,\ldots,a_d)$, the following terms will
be used.  The sequence is \emph{palindromic} if $a_k=a_{d-k}$ for
every $k$.  It is \emph{unimodal} if, for some $r$,
$a_0\leq\cdots\leq a_r\geq\cdots\geq a_d$.
It is \emph{log-concave} if $a_k^2\geq a_{k-1}a_{k+1}$ for
$1\leq k<d$, and \emph{strictly log-concave} if all these
inequalities are strict.  It is \emph{ultra log-concave in degree
$d$} if
\begin{equation}\label{eq:ulc-definition}
\left(\frac{a_k}{\binom dk}\right)^2
\geq
\frac{a_{k-1}}{\binom d{k-1}}
\frac{a_{k+1}}{\binom d{k+1}}
\qquad(1\leq k<d).
\end{equation}
A nonzero polynomial in $\RR[y]$ is \emph{real-rooted} if all its
complex zeros are real, with multiplicities included.
Xin and Zhang call \eqref{eq:ulc-definition} \emph{strong log-concavity}~\cite[Section~5]{XZ}.  We keep its degree
normalization visible.  
It is clear that \eqref{eq:ulc-definition} is equivalent to the \emph{Newton inequalities}
\begin{equation*}
a_k^2\geq\frac{(k+1)(d-k+1)}{k(d-k)}a_{k-1}a_{k+1}.
\end{equation*}

A polynomial $A(y)\in\RR[y]$ satisfying $y^{2e}A(1/y)=A(y)$ has a
\emph{gamma expansion}
\[A(y)=\sum_{j=0}^e\gamma_jy^j(1+y)^{2e-2j}.
\]
The expansion is unique.
It is \emph{gamma-positive} if $\gamma_j\geq0$ for all $j$, and
\emph{strictly gamma-positive} if $\gamma_j>0$ for all $j$.
The auxiliary polynomial $\Gamma(s)=\sum_{j=0}^e\gamma_js^j$ is
called its \emph{gamma polynomial}.

\subsection{Main results}

To resolve \cref{conj:xz}, we need to introduce the concepts of real stability and symmetric polynomials.

Set $\mathbb H=\{z\in\CC:\operatorname{Im}z>0\}$ and $\overline{\mathbb H}=\{z\in\CC:\operatorname{Im}z \geq 0\}$.
A nonzero polynomial $F\in\CC[x_1,\ldots,x_q]$ is \emph{stable} if
\[F(x_1,\ldots,x_q)\ne0 \quad\text{whenever }(x_1,\ldots,x_q)\in\mathbb H^q.\]
It is called \emph{real stable} if, in addition, its coefficients are real.
Nonzero constants are stable; zero is excluded.

Let $\CC[x_1,\ldots,x_q]^{\mathfrak S_q}$ be the ring of polynomials invariant under permutations of the variables, and set
\[
\mathcal S_{q,\leq R}
=\{F\in\CC[x_1,\ldots,x_q]^{\mathfrak S_q}:\deg F\leq R\}.
\]
On $\CC[x_1,\ldots,x_q]^{\mathfrak S_q}$, write $\partial_i=\partial/\partial x_i$ and define
\begin{equation*}
\mathcal A_q=\sum_{i=1}^q x_i^2\partial_i^2,\qquad
\mathcal B_q=2\sum_{1\leq i<j\leq q}\frac{x_i^2\partial_i-x_j^2\partial_j}{x_i-x_j},\qquad
\mathcal D_q=\mathcal B_q-\mathcal A_q.
\end{equation*}
The operator $\mathcal D_q$ is twice the Laplace--Beltrami operator at parameter $-1$ in Stanley's normalization \cite[p.~84, (11)]{StanleyJack}.  

The first main result of this paper is as follows.
\begin{theorem}
For $t\geq0$ and $R\in\NN$, the operator $e^{t\mathcal D_q}$ on $\mathcal S_{q,\leq R}$ preserves stability and real stability.
\end{theorem}

A partition of length at most $q$ is written $\lambda=(\lambda_1,\ldots,\lambda_q)$ with $\lambda_1\geq\cdots\geq\lambda_q\geq0$, including trailing zeros.
For partitions of equal size, the dominance relation $\mu\preceq\lambda$ means
\begin{equation*}
\sum_{i=1}^k\mu_i\leq\sum_{i=1}^k\lambda_i \qquad(1\leq k\leq q).
\end{equation*}
Write $\mu\prec\lambda$ if also $\mu\ne\lambda$.
The \emph{monomial symmetric polynomial} $m_\lambda$ is the sum of the distinct monomials whose exponent vectors are permutations of
$\lambda$.  
Let $\mathcal V_\lambda=\operatorname{span}_{\CC}\{m_\mu:\mu\preceq\lambda\}$.

The second main result of this paper is stated as follows.
\begin{theorem}
Let $q\geq1$ and let $\lambda$ have length at most $q$ and satisfy $\lambda_i-\lambda_{i+1}\leq1$ for $1\leq i<q$.  
There is a unique polynomial
$P_\lambda\in\mathcal V_\lambda$ of the form
\begin{equation*}
P_\lambda=m_\lambda+\sum_{\mu\prec\lambda}p_{\lambda\mu}m_\mu
\end{equation*}
satisfying
\begin{equation}\label{Equation-Introd-DQP}
\mathcal D_qP_\lambda=\left(\sum_{i=1}^q\lambda_i\bigl(2(q-i)-\lambda_i+1\bigr)\right) P_\lambda.
\end{equation}
It has real coefficients and is real stable.  Consequently, every polynomial in $\mathcal V_\lambda$ with a nonzero $m_\lambda$
coefficient and satisfying \eqref{Equation-Introd-DQP} is stable.
\end{theorem}

For the formulas to read more naturally, we set $m=n-2$.
For the convenience of the subsequent presentation, we also set
\begin{equation}\label{eq:parameters}
N=\left\lfloor\frac{(m+1)^2}{4}\right\rfloor, \qquad e=\left\lfloor\frac{m^2}{4}\right\rfloor.
\end{equation}
Their relations are $L:=2N+2e=m(m+1)$ and $N-m=\left\lfloor\frac{(m-1)^2}{4}\right\rfloor$.
Indeed, if $m=2r$, then $(N,e)=(r(r+1),r^2)$, while if $m=2r+1$, then $(N,e)=((r+1)^2,r(r+1))$. 
We shall write
$$A_m(y)=y^{-N}h_{m+2}^*(y),$$ initially as an element of $\mathbb Q((y))$.

The third main result of this paper is the following theorem, thereby proving \cref{conj:xz} of Xin and Zhang.

\begin{theorem}\label{thm:resolution}
Let $m\geq1$ and use the parameters in \eqref{eq:parameters}.
The series $h_{m+2}^*(y)$ in \eqref{eq:hstar-definition} has the
following properties.
\begin{enumerate}
\item[(i)] There are positive integers $\gamma_{m,0},\ldots,\gamma_{m,e}$ such that
\begin{equation}\label{eq:gamma-main}
h_{m+2}^*(y)=y^N\sum_{j=0}^e
\gamma_{m,j}y^j(1+y)^{2e-2j}.
\end{equation}
They are given by the finite formula \eqref{eq:gamma-flow-formula}.  Consequently,
\begin{equation}\label{eq:reduced-definition}
h_{m+2}^*(y)=y^NA_m(y),\qquad
A_m(y)=\sum_{k=0}^{2e}a_{m,k}y^k\in\ZZ[y],
\end{equation}
where every $a_{m,k}$ is positive.  In particular the order at zero of $h_{m+2}^*(y)$ is exactly $N$, and its degree is $N+2e$.

\item[(ii)] Every zero of $A_m(y)$ is a negative real number. Thus $h_{m+2}^*(y)$ has a zero of multiplicity $N$ at zero and all
its other zeros are negative.

\item[(iii)] For fixed $m$, the reduced coefficients $a_{m,k}$ are palindromic and, when $e\geq1$, increase strictly to their middle coefficient:
\[a_{m,k}=a_{m,2e-k}\quad(0\leq k\leq2e),\qquad 0<a_{m,0}<a_{m,1}<\cdots<a_{m,e}.\]
They satisfy the precise Newton inequalities
\begin{equation}\label{eq:main-newton}
a_{m,k}^2\geq \frac{(k+1)(2e-k+1)}{k(2e-k)}a_{m,k-1}a_{m,k+1} \qquad(1\leq k<2e).
\end{equation}
\end{enumerate}
For $m=1$, the assertions read $h_3^*(y)=y$ and $A_1(y)=1$.
\end{theorem}

The paper is organized as follows. The proof separates coefficient positivity from real-rootedness.
In \cref{sec:ct}, Xin's formal residue theorem \cite[Theorem~3-2.7]{XinThesis} converts the generating series
into a polynomial constant term.  
In \cref{sec:gamma}, the constant term's expansion yields a finite flow formula for the gamma coefficients. 
This proves part~\textup{(i)} and the symmetry and strict unimodality in part~\textup{(iii)}.
\cref{st:section} is mainly devoted to extensive preparations for the subsequent proof of real-rootedness.
In this section, we discuss the real stability of symmetric polynomials and the preservation of real stability of polynomials under certain partial differentiation operations.
In \cref{ct:section}, we introduce a symmetric polynomial $F_{m,q}$ with $q\geq\max(1,2m-2)$ variables. We use it to complete the proof of real-rootedness. We also proved \eqref{eq:main-newton} using Newton's inequalities.
\cref{sec:consequences} presents a byproduct: we prove that every root of the gamma polynomial is a negative real number.

\section{Formal constant terms}\label{sec:ct}

\subsection{The working field}

We use Xin's field of iterated Laurent series \cite[Sections~2-1 and~2-2]{XinThesis}.  The same framework is used in \cite[Section~2]{XinFast} and \cite{XinEuclid}.  Empty products equal $1$ and empty sums equal $0$.

Let $K$ be a field of characteristic zero.  Starting with
$\mathcal L_0=K$, define recursively
\begin{equation}\label{eq:laurent-tower}
\mathcal L_r=K((z_1))\cdots((z_r))
             =\mathcal L_{r-1}((z_r))\qquad(r\geq1).
\end{equation}
The displayed tower fixes the order throughout.  It agrees with
\cite[Section~2-1]{XinThesis} and \cite[Section~2]{XinFast}.
For any coefficient field $\mathbb{E}$, the notation $\mathbb{E}((z))$ means the field
of series $\sum_{k\geq k_0}c_kz^k$, where $k_0\in\ZZ$ and
$c_k\in \mathbb{E}$.  Thus an element of $\mathcal L_r$ is first a Laurent
series in $z_r$, bounded below in its $z_r$ exponent, whose
coefficients are elements of $\mathcal L_{r-1}$.  The lower bounds
for exponents of earlier variables are allowed to depend on later
exponents.  For example,
\[
\frac1{z_1-z_2}=\sum_{k\geq0}z_1^{-k-1}z_2^k
               \in K((z_1))((z_2)).
\]
Its $z_1$ exponents have no uniform lower bound.  This is why a ring
in which the exponents of every variable have a common lower bound
would be insufficient here.

Every rational function in $K(z_1,\ldots,z_r)$ has a unique image
in \eqref{eq:laurent-tower}: polynomials are embedded in the
natural way and each nonzero polynomial is inverted in this field.
The field, and not just the rational function, specifies the
expansion.  For instance, the expansion of $(z_1-z_2)^{-1}$ in
$K((z_2))((z_1))$ would instead be
$-\sum_{k\geq0}z_1^kz_2^{-k-1}$.

Here is the support description of this convention.  For $\alpha,\beta\in\ZZ^r$, by comparing the largest index $j$ at which $\alpha_j\ne\beta_j$, we declare $\alpha<_{\mathrm{rev}}\beta$ when $\alpha_j<\beta_j$.
A subset of $\ZZ^r$ is \emph{well ordered} in this order if every nonempty subset has a least element.  The \emph{support} of
$F=\sum_\alpha c_\alpha z_1^{\alpha_1}\cdots z_r^{\alpha_r}$
is the set of $\alpha$ for which $c_\alpha\ne0$. Xin's fundamental structure theorem
\cite[Proposition~2-1.2, p.~60]{XinThesis} states that $F$ belongs to $\mathcal L_r$ if and only if its support is well ordered.
The composition law \cite[Theorem~3-1.7, p.~113]{XinThesis} then implies that, if the least exponent of $U$ is greater than
$(0,\ldots,0)$, the series $\sum_{k\geq0}c_kU^k$ is defined in $\mathcal L_r$ for arbitrary $c_k\in K$.  Its support is well
ordered, and every coefficient receives only finitely many contributions.  In particular, factoring a nonzero polynomial
into its initial term and a factor with initial term $1$ yields its reciprocal by geometric expansion.  
These facts justify the rational expansions used below.

To include the generating variable, we use
\begin{equation}\label{eq:laurent-generating-field}
\mathcal L_r((y))=K((z_1))\cdots((z_r))((y)).
\end{equation}
In particular, $y$ is the outermost Laurent variable.  The phrase ``$y$ is smaller than all the $z_i$'' will mean this specified
formal order, without an assertion about complex absolute values.
Thus $z_j/z_i$ for $i<j$, $z_i$, and $y/z_i$ have positive least exponent in their respective working fields. 
Consequently, the expansions that occur here are
\begin{align}
\frac1{(z_i-z_j)^2}&=z_i^{-2}\sum_{f\geq0}(f+1)\left(\frac{z_j}{z_i}\right)^f&&(i<j),\label{eq:edge-expansion}\\
\frac1{(1-z_i)^a}&=\sum_{k\geq0}\binom{a+k-1}{k}z_i^k &&(a\geq1),\notag\\
\frac1{z_i-y}&=z_i^{-1}\sum_{k\geq0}\left(\frac y{z_i}\right)^k.\notag
\end{align}

For $F\in\mathcal L_r$, the notation $[z^\alpha]F$ denotes its coefficient with exponent vector $\alpha\in\ZZ^r$.
The \emph{constant term} and the \emph{formal residue} are
\[\CT_z F=[z_1^0\cdots z_r^0]F,\qquad
\Res_z F(z)\,\dd z_1\cdots\dd z_r=[z_1^{-1}\cdots z_r^{-1}]F.
\]
For a single variable, $\Res_{z_i}F\,\dd z_i$ means extraction of its coefficient of $z_i^{-1}$, leaving a series in the
remaining variables.  The differential symbols $\dd z_i$ record which coefficient is extracted; they do not denote an
integration operation.  On \eqref{eq:laurent-generating-field}, these operators act coefficientwise in $y$.

By taking a specified exponent in one variable, we restrict the support to a subset of a well-ordered set.  
After deleting that coordinate, the induced order is again the reverse lexicographic order.
Thus the result lies in the iterated Laurent series field in the remaining variables.  
Repeated coefficient extractions commute:
each order selects exactly the same coefficients.  These are Xin's constant-term operators
\cite[Definition~2-1.3 and properties~P1--P3, pp.~60--61]{XinThesis}.  In particular,
\begin{equation}\label{eq:formal-ct-residue}
\CT_z F
 =\Res_z\frac{F}{z_1\cdots z_r}\,\dd z_1\cdots\dd z_r.
\end{equation}

Note that differentiation is termwise differentiation of Laurent series.
It is a derivation on the working field: it satisfies the sum and product rules, and hence
$\partial_i(H^{-1})=-H^{-2}\partial_iH$ for nonzero $H$.
These assertions follow first in one Laurent variable, where every coefficient of a product is a finite sum, and then
recursively in \eqref{eq:laurent-tower}.  Every derivative $\partial_i=\partial/\partial z_i$ fixes $y$ and acts coefficientwise in $y$. Writing $\theta_i=z_i\partial_i$, we have
\begin{equation}\label{eq:formal-exact-residue}
\Res_z(\partial_iF)\,\dd z_1\cdots\dd z_r=0,
\qquad
\CT_z(\theta_iF)=0.
\end{equation}
For the first equality, a derivative could contribute exponent
$-1$ in $z_i$ only by differentiating a term with exponent $0$;
that derivative is zero.  For the second equality, $\theta_i$
multiplies a monomial by its $z_i$ exponent, which is zero at a
constant term.  This is the formal integration-by-parts rule of
\cite[Section~2-2, Rule~3, p.~66]{XinThesis}.

\subsection{An admissible substitution}

The following two lemmas are derived from Xin's composition lemma and residue theorem \cite[Lemma~3-2.6 and Theorem~3-2.7, pp.~123--125]{XinThesis}.
For clarity, the result is stated in constant-term form in \cite{XinThesis}, with a logarithmic Jacobian.  
By applying it to $x_1\cdots x_r yF(x,y)$, we can convert it to the ordinary residue formula by \eqref{eq:formal-ct-residue}.  
Therefore, we omit the proofs of these two lemmas.

\begin{lemma}\label{lem:formal-translation}
Let $F\in K((x_1))\cdots((x_r))((y))$.  The substitution
$x_i=v_i+y$ is defined by its Taylor expansion in $y$ and gives
an isomorphism to $K((v_1))\cdots((v_r))((y))$.
It preserves formal residues:
\begin{equation*}
\Res_x F(x,y)\,\dd x_1\cdots\dd x_r=\Res_v F(v_1+y,\ldots,v_r+y,y)\,\dd v_1\cdots\dd v_r.
\end{equation*}
For rational $F$, the substituted expression on the right is the
unique expansion of the indicated rational function in the stated
working field.
\end{lemma}

\begin{lemma}\label{lem:formal-origin-substitution}
Set $a=1-y$.  For $F\in K((v_1))\cdots((v_r))((y))$, the substitution
\[
v_i=\frac{az_i}{1+z_i}\qquad(1\leq i\leq r)
\]
has a well-defined expansion in
$K((z_1))\cdots((z_r))((y))$.  With this interpretation, we have 
\begin{equation*}
\Res_vF(v,y)\,\dd v_1\cdots\dd v_r
=\Res_zF\left(\frac{az_1}{1+z_1},\ldots,\frac{az_r}{1+z_r},y\right)\prod_{i=1}^r\frac{a}{(1+z_i)^2}\,\dd z_1\cdots\dd z_r.
\end{equation*}
For rational functions, this interpretation agrees with their unique expansion in the stated field.
\end{lemma}

\begin{corollary}\label{cor:formal-mobius}
For $F\in K(x_1,\ldots,x_r,y)$, expand $F$ in $K((x_1))\cdots((x_r))((y))$.  Then
\begin{equation*}
\Res_xF(x,y)\,\dd x_1\cdots\dd x_r=\Res_z F\left(\frac{z_1+y}{1+z_1},\ldots,\frac{z_r+y}{1+z_r},y\right)\cdot\prod_{i=1}^r\frac{1-y}{(1+z_i)^2}\,\dd z_1\cdots\dd z_r,
\end{equation*}
where the right-hand side is expanded in $K((z_1))\cdots((z_r))((y))$.
\end{corollary}
\begin{proof}
We first apply \cref{lem:formal-translation} with $x_i=v_i+y$ and then use \cref{lem:formal-origin-substitution} with
$v_i=(1-y)z_i/(1+z_i)$. 
Their composite sends
$$x_i\longmapsto y+\frac{(1-y)z_i}{1+z_i}=\frac{z_i+y}{1+z_i}.$$
The first substitution has derivative one, and the second supplies the displayed factor $(1-y)/(1+z_i)^2$.  
Both lemmas specify the same outer variable $y$. Thus, their expansion conventions agree.  
This proves the identity in the asserted working fields.
\end{proof}

\subsection{A finite constant term functional}

For $r\geq1$, set
\[
\Delta_r(z_1,\ldots,z_r)=\prod_{1\leq i<j\leq r}(z_i-z_j),
\qquad \Delta_1=1.
\]
The functional used throughout the paper is
\begin{equation}\label{eq:T-definition}
T_m(f)=\CT_z\frac{f(z)}{\Delta_m(z)^2},
\end{equation}
with the rational function expanded in
$K((z_1))\cdots((z_m))$.  When $f$ also has polynomial
parameters, $T_m$ acts coefficientwise in those parameters.  Thus
an auxiliary alphabet such as $b_1,\ldots,b_q$, or a homogenizing
variable $w$, is not included among the variables whose constant
term is extracted.  Equivalently, for such a polynomial one may
extend $K$ by adjoining the auxiliary parameters and take the
same $z$ constant term.
The following description shows directly that its value on a
polynomial is \textbf{a finite sum}.

\begin{definition}\label{def:flows}
The directed graph used here has vertex set
$\{1,\ldots,m\}$ and one directed edge $i\to j$ whenever
$i<j$.  A \emph{nonnegative integer flow} on this graph is
simply an assignment of an integer $f_{ij}\geq0$ to every
edge.  Its \emph{netflow} at vertex $i$ is its outgoing amount
minus its incoming amount,
\[
\sum_{j>i}f_{ij}-\sum_{j<i}f_{ji}.
\]
A \emph{prefix cut} at $k$ separates the vertices
$\{1,\ldots,k\}$ from $\{k+1,\ldots,m\}$; its flow is $\sum_{i\leq k<j}f_{ij}$.
Conservation of flow is not assumed: the netflows are determined by the coefficient equations below.
\end{definition}

We now consider $T_m(z^d)$.
For a monomial $z^d=z_1^{d_1}\cdots z_m^{d_m}$, multiplication of the edge expansions \eqref{eq:edge-expansion} yields exponent
\begin{align*}
d_i-2(m-i)-\sum_{j>i}f_{ij}+\sum_{j<i}f_{ji}\quad \text{ at } z_i.
\end{align*}
The term contributes to $T_m(z^d)$ exactly when all
these exponents are zero, that is, when
\begin{equation}\label{eq:monomial-balance}
d_i=2(m-i)+\sum_{j>i}f_{ij}-\sum_{j<i}f_{ji}\qquad(1\leq i\leq m).
\end{equation}
For such a flow, its contribution is the positive integer
$\prod_{i<j}(f_{ij}+1)$.  Summing
\eqref{eq:monomial-balance} over $i=1,\ldots,k$ cancels every
edge entirely inside the prefix.  Each crossing edge occurs once
with positive sign, and each edge outside the prefix does not
occur.  Since
$\sum_{i=1}^k2(m-i)=2km-k(k+1)$, we obtain
\begin{equation}\label{eq:general-cut}
\sum_{i\leq k<j}f_{ij}
 =\sum_{i=1}^k d_i-2km+k(k+1)\qquad(1\leq k<m).
\end{equation}
For fixed $d$, the right-hand side is fixed.  If it is negative
for any $k$, there is no contributing flow.  Otherwise, for $i<j$, the edge crosses the prefix cut at $k=i$. Thus, we have 
\[
0\leq f_{ij}\leq\sum_{p\leq i<q}f_{pq}=\sum_{p=1}^i d_p-2im+i(i+1).
\]
Every edge therefore has a fixed finite upper bound, and there
are only $\binom m2$ edges. \textbf{This proves finiteness}.  
For $m=1$, there are no edges, and the single possible empty flow contributes precisely when $d_1=0$, giving the same conclusion.  
By linearity, $T_m$ of any polynomial is a finite linear combination of its coefficients.  
In particular it sends a polynomial with nonnegative integral coefficients to a nonnegative integer, coefficientwise
in any additional polynomial parameters.

\subsection{A simple constant term form}

We now work over $K=\mathbb Q$ and use the original expansion field $\mathbb Q((x_2))\cdots((x_n))$ in \eqref{eq:original}.  
Relabel $x_2,\ldots,x_{n-1}$ as $x_1,\ldots,x_m$ and set $u=x_n$.  
There are two cancellations in \eqref{eq:original}: the numerator factor $1-u$ changes $(1-u)^{-3}$ to $(1-u)^{-2}$, and each numerator factor $x_i-u$ cancels one of the two corresponding denominator factors.  
Therefore
\begin{equation}\label{eq:original-cancelled}
h_{m+2}(t)
 =[u^t]\frac1{(1-u)^2}\CT_x
 \frac{\prod_{i=1}^m x_i^{m+1}}
      {\prod_{i=1}^m(1-x_i)^3(x_i-u)\,\Delta_m(x)^2}.
\end{equation}
Here $[u^t]$ extracts an ordinary nonnegative power: in
$\mathbb Q((x_1))\cdots((x_m))((u))$ each factor
$(x_i-u)^{-1}=\sum_{k\geq0}x_i^{-k-1}u^k$ has no negative
powers of $u$, and neither does $(1-u)^{-2}$.
Thus the entire expression before $[u^t]$ lies in
$\mathbb Q[[u]]$ after its $x$ constant term is taken.

By replacing $u$ by $y$ in this formal power series and summing its coefficients, we obtain
\begin{equation}\label{eq:generating-ct}
\sum_{t\geq0}h_{m+2}(t)y^t=\frac1{(1-y)^2}\CT_x
 \frac{\prod_{i=1}^m x_i^{m+1}}{\prod_{i=1}^m(1-x_i)^3(x_i-y)\,\Delta_m(x)^2}.
\end{equation}
This equality lies in $\mathbb Q[[y]]$.  The constant term on
its right is taken in $\mathbb Q((x_1))\cdots((x_m))((y))$ and acts coefficientwise
in $y$.  At each $y$ exponent, only finitely many exponent choices in the
$m+1$ factors involving $y$ occur.  The sign agrees with
\eqref{eq:original}, which already includes the normalization
in \eqref{eq:birkhoff-component}.

\begin{proposition}\label{prop:positive-ct}
Let $m\geq1$.  With constant terms taken in the field
$\mathbb Q((z_1))\cdots((z_m))((y))$, one has
\begin{equation}\label{eq:positive-ct}
h_{m+2}^*(y)=\CT_z\frac{\prod_{i=1}^m(z_i+y)^m(1+z_i)^m}{\Delta_m(z)^2}
=T_m\left(\prod_{i=1}^m\bigl(z_i^2+(1+y)z_i+y\bigr)^m\right).
\end{equation}
In particular, $h_{m+2}^*(y)$ is a polynomial with nonnegative
integral coefficients.
\end{proposition}
\begin{proof}
We use \eqref{eq:formal-ct-residue} to express \eqref{eq:generating-ct} as the formal residue
\begin{equation}\label{eq:ordinary-density}
\Res_x\frac1{(1-y)^2}\,\frac{\prod_{i=1}^m x_i^m}{\prod_{i=1}^m(1-x_i)^3(x_i-y)\,\Delta_m(x)^2}\,\dd x_1\cdots\dd x_m.
\end{equation}
The power $x_i^m$ instead of $x_i^{m+1}$ is due to the factor
$x_i^{-1}$ in \eqref{eq:formal-ct-residue}.
By \cref{cor:formal-mobius}, the admissible change is
\begin{equation}\label{eq:mobius-substitution}
x_i=\frac{z_i+y}{1+z_i}\qquad(1\leq i\leq m).
\end{equation}
All required identities are the following rational identities,
interpreted through their unique expansions in the working fields:
\begin{align*}
x_i^m=\frac{(z_i+y)^m}{(1+z_i)^m},\quad \dd x_i=\frac{1-y}{(1+z_i)^2}\,\dd z_i,\quad
1-x_i=\frac{1-y}{1+z_i},\quad
x_i-y=\frac{(1-y)z_i}{1+z_i},
\end{align*}
and 
\begin{align*}
x_i-x_j=\frac{(1-y)(z_i-z_j)}{(1+z_i)(1+z_j)}.
\end{align*}
In particular,
\[
\Delta_m(x)^2
 =(1-y)^{m(m-1)}\Delta_m(z)^2
       \prod_{i=1}^m(1+z_i)^{-2(m-1)}.
\]
Each index belongs to $m-1$ pairs, which explains its exponent
$-2(m-1)$ in this product.

For fixed $i$, the terms in \eqref{eq:ordinary-density} contribute to the exponent of $1+z_i$ as follows:
\[
x_i^m\mapsto -m,\quad
\dd x_i\mapsto -2,\quad
(1-x_i)^{-3}\mapsto 3,\quad
(x_i-y)^{-1}\mapsto 1,\quad
\Delta_m(x)^{-2}\mapsto 2(m-1).
\]
Their sum is $m$.  
The prefactor, differentials, the factors $(1-x_i)^{-3}$ and $(x_i-y)^{-1}$, and the squared Vandermonde contribute to the
exponent of $1-y$ as follows:
\[
-2+m-3m-m-m(m-1) =-m^2-2m-2=-\bigl((m+1)^2+1\bigr).
\]
Finally, $(x_i-y)^{-1}$ supplies one factor $z_i^{-1}$, and $x_i^m$ supplies $(z_i+y)^m$.  
Consequently, \eqref{eq:ordinary-density} becomes
\[\Res_z (1-y)^{-((m+1)^2+1)}\frac{\prod_{i=1}^m(z_i+y)^m(1+z_i)^m}{\Delta_m(z)^2}\prod_{i=1}^m\frac{\dd z_i}{z_i}.
\]
Its formal residue is the indicated constant term multiplied by the displayed power of $1-y$.
Since $n=m+2$, the exponent in \eqref{eq:hstar-definition} equals $(m+1)^2+1$; 
multiplication by this factor therefore proves the first equality in \eqref{eq:positive-ct}.
The identity $(z_i+y)(1+z_i)=z_i^2+(1+y)z_i+y$ proves the second.
The polynomial numerator has nonnegative integral coefficients, and the finite-flow argument following \eqref{eq:general-cut} applies coefficientwise in $y$. 
It follows that this constant term belongs to $\mathbb Z_{\geq0}[y]$. This proves the last assertion.
\end{proof}

\section{The gamma coefficients}\label{sec:gamma}

\subsection{Expression of the gamma coefficients}

Fix $m\geq1$.  We expand quotients of polynomials in
$\mathbb Q[y,z_1,\ldots,z_m]$ by $\Delta_m(z)^2$ in $\mathbb Q(y)((z_1))\cdots((z_m))$, and $T_m$ extracts their
$z$ constant terms.  Because the denominator is independent of
$y$ and the numerator is polynomial in $y$, this agrees with
coefficientwise extraction in
$\mathbb Q((z_1))\cdots((z_m))((y))$.

At each vertex $i\in\{1,\ldots,m\}$, we have 
\begin{equation}\label{eq:trinomial}
\bigl(z_i^2+(1+y)z_i+y\bigr)^m
=\sum_{\substack{a_i,b_i,c_i\geq0\\a_i+b_i+c_i=m}}\frac{m!}{a_i!b_i!c_i!}z_i^{2a_i+b_i}(1+y)^{b_i}y^{c_i}.
\end{equation}
With $f_{ij}\in\NN$ as in \eqref{eq:edge-expansion}, the denominator expands as
\begin{equation}\label{eq:flow-denominator}
\frac1{\Delta_m(z)^2}
=\left(\prod_{i=1}^m z_i^{-2(m-i)}\right)
\sum_{(f_{ij})\in\NN^{\binom m2}}
\left(\prod_{i<j}(f_{ij}+1)\right)
\prod_{i<j}\left(\frac{z_j}{z_i}\right)^{f_{ij}}.
\end{equation}
Regard $f=(f_{ij})_{i<j}$ as a flow in the directed graph of \cref{def:flows}.  The exponent of $z_i$ in a term
of the product of \eqref{eq:trinomial} and
\eqref{eq:flow-denominator} is
$2a_i+b_i-2(m-i)-\sum_{j>i}f_{ij}+\sum_{j<i}f_{ji}$.
Thus a term contributes precisely when
\begin{equation}\label{eq:balance}
d_i:=2a_i+b_i
=2(m-i)+\sum_{j>i}f_{ij}-\sum_{j<i}f_{ji}.
\end{equation}
We call $(f,a,b,c)$ \emph{admissible} if all its entries are nonnegative integers, \eqref{eq:balance} holds, and $a_i+b_i+c_i=m$ for every vertex.
Its weight is
\[
 \left(\prod_{i<j}(f_{ij}+1)\right)\left(\prod_i\frac{m!}{a_i!b_i!c_i!}\right).
\]
This weight is a positive integer.

\begin{lemma}\label{lem:finite-gamma}
There are only finitely many arrays $(f,a,b,c)$ of nonnegative
integers satisfying \eqref{eq:balance} and
$a_i+b_i+c_i=m$.  More precisely, if
$F_k=\sum_{i\leq k<j}f_{ij}$, then
\begin{equation}\label{eq:cut-bound}
F_k=\sum_{i=1}^k d_i-2km+k(k+1),\qquad
0\leq F_k\leq k(k+1)\quad(0\leq k\leq m).
\end{equation}
\end{lemma}
\begin{proof}
When \eqref{eq:balance} is summed over $1\le i\le k$, the internal flows cancel and the remaining flow is $F_k$.
This yields the equality in \eqref{eq:cut-bound}.
Since $0\leq d_i=2a_i+b_i\leq2m$, the right-hand side of \eqref{eq:cut-bound} is at most
$k(k+1)$; nonnegativity follows from the definition of $F_k$.
Each edge $f_{ij}$ crosses the cut at $k=i$. Thus $f_{ij}\leq F_i\leq i(i+1)$.  
The flows and the vertex triples therefore range over finite sets.  
The same argument includes the empty cuts $k=0,m$ and the case $m=1$.

At the last vertex, we have $0\leq d_m=-\sum_{i<m}f_{im}\leq0$.  
Hence $d_m=0$ and every incoming edge at $m$ has zero flow. 
Since $d_m=2a_m+b_m$ with nonnegative entries, we have $a_m=b_m=0$ and $c_m=m$.
\end{proof}

\cref{lem:finite-gamma} shows that \eqref{eq:positive-ct} is a polynomial of $y$.

For an admissible array, set
\[
\mathsf A=\sum_{i=1}^{m} a_i,\qquad
\mathsf B=\sum_{i=1}^{m} b_i,\qquad
\mathsf C=\sum_{i=1}^{m} c_i.
\]
Summing \eqref{eq:balance} over all vertices cancels the flows and yields
\begin{equation}\label{eq:global-sums}
2\mathsf A+\mathsf B=m(m-1),\qquad
\mathsf A+\mathsf B+\mathsf C=m^2.
\end{equation}
Consequently, we have 
\begin{equation}\label{eq:gamma-sums}
\mathsf C=m+\mathsf A,
\qquad \mathsf B=L-2\mathsf C.
\end{equation}
An admissible array therefore contributes its weight times
\begin{equation}\label{eq:one-gamma-term}
y^{\mathsf C}(1+y)^{L-2\mathsf C}.
\end{equation}

\begin{lemma}\label{lem:range}
Every admissible array satisfies $N\leq\mathsf C\leq N+e$.
\end{lemma}

\begin{proof}
Since $d_i=a_i+(a_i+b_i)$ and $a_i+b_i\leq m$, one has
$a_i\geq d_i-m$.  For each $0\leq k\leq m$, it follows that
\begin{align*}
\mathsf A\geq\sum_{i=1}^k a_i\geq\sum_{i=1}^k d_i-km=2km-k(k+1)+F_k-km=k(m-k-1)+F_k\geq k(m-k-1).
\end{align*}
Since $k(m-k-1)=\frac{(m-1)^2}{4}-\left(k-\frac{m-1}{2}\right)^2$, the maximum over integral $0\leq k\leq m$ is
$\lfloor(m-1)^2/4\rfloor=N-m$. We have $\mathsf C=m+\mathsf A\geq N$.  The upper bound follows
from $\mathsf B=L-2\mathsf C\geq0$ and $L/2=N+e$.
\end{proof}

By \cref{lem:finite-gamma,lem:range}, we may
group the expansion in \cref{prop:positive-ct} by
$j=\mathsf C-N$.  This proves \eqref{eq:gamma-main}, with
\begin{equation}\label{eq:gamma-flow-formula}
\gamma_{m,j}=\sum_{\substack{f_{ij},a_i,b_i,c_i\in\NN\ \text{satisfying}\ \eqref{eq:balance}\\
                  a_i+b_i+c_i=m\ (1\leq i\leq m)\\
                  \sum_i c_i=N+j}}
\left(\prod_{i<j}(f_{ij}+1)\right)\left(\prod_{i=1}^m\frac{m!}{a_i!b_i!c_i!}\right).
\end{equation}
Indeed, $L-2\mathsf C=2e-2j$.  Thus, the gamma coefficients are nonnegative integers.  
We next show that every indexing set in \eqref{eq:gamma-flow-formula} is nonempty.

\subsection{Positivity and unimodality}

\begin{lemma}\label{lem:integer-intervals}
If $\ell_i,u_i$ are integers with $\ell_i\leq u_i$, $1\leq i\leq m$, then the
set of sums $\sum_{i=1}^m a_i$, with independently chosen integers $\ell_i\leq a_i\leq u_i$, is the entire integer interval
$[\sum_{i=1}^m \ell_i,\sum_{i=1}^m u_i]$.
\end{lemma}
\begin{proof}
Clearly, any choice of integers $a_i \in [\ell_i, u_i]$ ($1 \le i \le m$) satisfies $\sum_{i=1}^m \ell_i \le \sum_{i=1}^m a_i \le \sum_{i=1}^m u_i$. 
Conversely, let $t$ be an integer in this interval, and set $s = t - \sum_{i=1}^m \ell_i$. 
Define the partial sums $D_k = \sum_{i=1}^k (u_i - \ell_i)$ for $0 \le k \le m$ (with $D_0 = 0$). 
Thus, $0 = D_0 \le s \le D_m$. 
Let $k \in \{1, \dots, m\}$ be the smallest index such that $s \le D_k$. 
By the minimality of $k$ (and $D_0 = 0 \le s$), we have $D_{k-1} \le s \le D_k$, and hence $0 \le s - D_{k-1} \le u_k - \ell_k$. Defining
\[
a_i = \begin{cases}
u_i, & 1 \le i < k, \\
\ell_k + s - D_{k-1}, & i = k, \\
\ell_i, & k < i \le m.
\end{cases}
\]
Therefore, we have $\ell_i \le a_i \le u_i$ for every $1 \le i \le m$, and
\[
\sum_{i=1}^m a_i = \sum_{i=1}^m \ell_i + \sum_{i=1}^{k-1} (u_i - \ell_i) + (s - D_{k-1}) = \sum_{i=1}^m \ell_i + s = t.
\]
This completes the proof.
\end{proof}

\begin{proposition}\label{prop:strict-gamma}
For every $0\leq j\leq e$, the sum
\eqref{eq:gamma-flow-formula} has a positive contribution with
$f_{ij}=0$ for every edge.  Thus $\gamma_{m,j}>0$.
\end{proposition}
\begin{proof}
With all flows zero, \eqref{eq:balance} becomes $2a_i+b_i=2(m-i)$.  
By solving this equation together with $a_i+b_i+c_i=m$, we have $b_i=2(m-i)-2a_i$ and $c_i=2i-m+a_i$.
The three variables are nonnegative precisely when
\begin{equation}\label{eq:zero-flow-intervals}
\max(0,m-2i)\leq a_i\leq m-i.
\end{equation}
These intervals are nonempty because $0\leq m-i$ and $m-2i\leq m-i$.

Their lower endpoints sum to $N-m$.  Explicitly, for $m=2r$ the positive summands are $2r-2,2r-4,\ldots,2$, whose sum is
$r(r-1)=N-m$; for $m=2r+1$ they are $2r-1,2r-3,\ldots,1$, whose sum is $r^2=N-m$.
Empty sums are understood as zero. Their upper endpoints sum to
\[\sum_{i=1}^m(m-i)=\frac{m(m-1)}2.
\]
By \cref{lem:integer-intervals}, every integer $\mathsf A$ between these two endpoint sums is realized.
Since $\mathsf C=m+\mathsf A$, the realized values of
$\mathsf C$ form exactly
\[
\left[N,\ m+\frac{m(m-1)}2\right]
=[N,L/2]=[N,N+e].
\]
Let $\mathsf C=N+j$.  The resulting array has edge weight one and positive multinomial weight, and hence contributes
positively to \eqref{eq:gamma-flow-formula}.
\end{proof}

\begin{proposition}\label{prop:gamma-structure}\label{thm:gamma}
The gamma coefficients in \eqref{eq:gamma-flow-formula} are positive integers, and \eqref{eq:gamma-main} holds in
$\mathbb Z[y]$.  The quotient $A_m(y)=y^{-N}h_{m+2}^*(y)$
is a polynomial of degree $2e$ with positive palindromic
coefficients.  These coefficients increase strictly through
index $e$ when $e\geq1$.  The order at zero of $h_{m+2}^*(y)$ is
$N$ and its degree is $N+2e$.
\end{proposition}
\begin{proof}
The gamma expansion and its positivity have been proved.
By expanding the powers of $1+y$ gives, with $\binom uv=0$ unless $0\leq v\leq u$, we obtain 
\begin{equation}\label{eq:ordinary-from-gamma}
a_{m,k}=\sum_{j=0}^e\gamma_{m,j}\binom{2e-2j}{k-j}\qquad(0\leq k\leq2e).
\end{equation}
The $j=0$ term is $\gamma_{m,0}\binom{2e}k>0$. This proves positivity.  
The identity $\binom{2e-2j}{2e-k-j}=\binom{2e-2j}{k-j}$, valid also
when the entries vanish. This proves palindromicity.  
Both endpoint coefficients are $\gamma_{m,0}>0$. This proves the asserted order and degree.

Fix $0\leq k<e$. We have
$$a_{m,k+1}-a_{m,k}= \sum_{j=0}^e\gamma_{m,j} \left(\binom{2e-2j}{k+1-j}-\binom{2e-2j}{k-j}\right).$$
Terms with $j>k+1$ vanish at both indices
$k,k+1$, and the term with $j=k+1$ increases from zero to
$\gamma_{m,k+1}$.  For $j\leq k$, set $q=2e-2j$ and $r=k-j$.  Then $r<e-j=q/2$, and
\[\binom q{r+1}-\binom qr =\binom qr\left(\frac{q-r}{r+1}-1\right)\geq0.
\]
For $j=0$ this difference is positive, because $(2e-k)/(k+1)>1$ when $k<e$.  
Since $\gamma_{m,j} > 0$ for all $0 \le j \le e$, summing over $j$ yields $a_{m,k+1} > a_{m,k}$ for every $0 \le k < e$. Thus, $a_{m,0} < a_{m,1} < \cdots < a_{m,e}$. By palindromicity, it follows that $a_{m,e} > a_{m,e+1} > \cdots > a_{m,2e}$.
Finally, for $m=1$ the constant term condition forces $a_1=b_1=0$ and $c_1=1$. Thus, $h_3^*(y)=y$.  
This proves part~\textup{(i)} of \cref{thm:resolution} and the symmetry and strict unimodality in part~\textup{(iii)}.
\end{proof}

\section{Real stable symmetric polynomial}
\label{st:section}

We construct a real stable symmetric eigenfunction of $\mathcal D_q$ with a prescribed leading monomial. 
The constant-term polynomial in the next section will be identified by this characterization.
Throughout this section $q\geq1$.  All polynomial limits are coefficientwise with a fixed degree bound, and all operator
exponentials act on finite-dimensional polynomial spaces.

\subsection{Stability and coefficient limits}

A real univariate polynomial is stable precisely when all its zeros are real, since nonreal zeros occur in conjugate pairs.  For complex univariate polynomials, we distinguish stability from the assertion that the zeros themselves belong to $\mathbb H$.

\begin{lemma}\label{st:root-continuity}
In the coefficient space $\CC^r$ of monic degree-$r$ polynomials,
where $r\geq1$, the set $\mathcal C_r^+$ of polynomials whose zeros
belong to $\overline{\mathbb H}$ is closed.  Its interior consists
precisely of those whose zeros belong to $\mathbb H$.
\end{lemma}
\begin{proof}
We first establish the required continuity of roots.  
Suppose that monic degree-$r$ polynomials $p_j$ converge coefficientwise to $p$. 
By coefficientwise convergence, there exists a constant $C > 0$ such that every nonleading coefficient of each $p_j$ has modulus at most $C$.
Since for $|z|>1+C$, 
\[\sum_{k=0}^{r-1}C|z|^k=C\frac{|z|^r-1}{|z|-1}<|z|^r.
\]
Hence $p_j(z)\neq 0$ whenever $|z|>1+C$. Thus, every zero of every $p_j(z)$ lies in the closed disk $|z|\leq 1+C$.
Therefore, every zero has modulus at most $1+C$. 
Write the roots as $\rho_{j,1},\ldots,\rho_{j,r}$.  
Every subsequence of these $r$-tuples has a convergent further subsequence.
If its coordinate limits are $\rho_1,\ldots,\rho_r$, then
\[
p(z)=\lim_j\prod_{i=1}^r(z-\rho_{j,i})=\prod_{i=1}^r(z-\rho_i).
\]
Thus the limiting multiset is exactly the root multiset of $p$.

In particular, every neighborhood of a root of $p$ contains a root of every sufficiently large $p_j$; 
and every open set containing all roots of $p$ eventually contains all roots of $p_j$.  
Otherwise the same compactness argument would either factor $p$ using roots outside a neighborhood of one of its roots, 
or produce a root of $p$ outside an open set containing all its roots.
This proves continuity of the root multiset in the form used here.

Closedness of $\mathcal C_r^+$ follows by taking root limits: 
If every root of each $p_j$ belongs to the closed upper half-plane, the preceding observation puts every root of $p$ there as well.
Thus $\mathcal C_r^+$ is closed.

If every root of $p$ belongs to $\mathbb H$, its finite root multiset has positive distance from the real axis.  
Root continuity yields a coefficient neighborhood of $p$ in which all roots remain in $\mathbb H$. 
Hence, $p$ is an interior point of $\mathcal C_r^+$.  
Conversely, if $p$ has a real root $a$, write $p(z)=(z-a)u(z)$.  
For $\delta>0$, replacing the factor $z-a$ by $z-a+\mathrm i\delta$ produces a monic degree $r$ polynomial with a root $a-\mathrm i\delta$ in the open lower half-plane; as $\delta\to 0^+$, its coefficients converge to those of $p$. Hence $p$ is not an interior point of $\mathcal{C}_r^+$.
\end{proof}

\begin{lemma}\label{st:root-persistence}
Let $R\in\NN$, and let polynomials $p_j(z)\in\CC[z]$ of degrees at most $R$ converge coefficientwise to a nonzero polynomial $p(z)$.
If $p(\alpha)=0$, every neighborhood of $\alpha$ contains a zero of every sufficiently large $p_j(z)$.
\end{lemma}
\begin{proof}
The case $R=0$ is vacuous. 
Since $p$ is nonzero, we may choose $w\in\mathbb{C}$ such that $p(w)\ne 0$ and set
\begin{equation}\label{st:reciprocal-polynomial}
Q_j(z)=z^R p_j(w+1/z),\qquad Q(z)=z^R p(w+1/z).
\end{equation}
These are polynomials: a term $a_k(w+1/z)^k$ contributes $a_kz^{R-k}(1+wz)^k$.  
Their coefficients of $z^R$ are $p_j(w)$ and $p(w)$, respectively.  
Consequently, $Q_j(z)/p_j(w)$ eventually has degree $R$ and converges to the monic degree $R$ polynomial $Q(z)/p(w)$, 
even if the degrees of $p_j$ vary.

The number $\beta=1/(\alpha-w)\ne0$ is a root of $Q(z)$.
By \cref{st:root-continuity}, roots $\beta_j$ of $Q_j(z)$ can be chosen with $\beta_j\to\beta$.  
They are eventually nonzero. By \eqref{st:reciprocal-polynomial}, we have $p_j(w+1/\beta_j)=0$ with $w+1/\beta_j\to\alpha$.
\end{proof}

\begin{lemma}\label{st:limit-closure}
Let $F_j$ be stable polynomials in a fixed number of variables, of
uniformly bounded total degree.  If they converge coefficientwise to
$F$, then $F$ is either stable or identically zero.
\end{lemma}
\begin{proof}
Suppose that the limit $F$ is nonzero but $F(a)=0$ for some
$a\in\mathbb H^q$.  There is $b\in\mathbb H^q$ with $F(b)\ne0$:
a polynomial vanishing on the product of open sets
$\mathbb H^q$ is zero, by successive applications of the
fact that a nonzero univariate polynomial has finitely many roots
to its coefficient polynomials.
For the given stable polynomials $F_j$, set
\[
g_j(t)=F_j\bigl(a+t(b-a)\bigr),\qquad
g(t)=F\bigl(a+t(b-a)\bigr).
\]
Then $g_j\to g$ coefficientwise with bounded degree, while
$g(0)=0$ and $g(1)\ne0$. We may choose $\eta>0$ such that
\[
\eta|b_i-a_i|<\operatorname{Im}a_i\quad\text{whenever }b_i\ne a_i.
\]
For $|t|<\eta$, the inequality
\[
\operatorname{Im}\bigl(a_i+t(b_i-a_i)\bigr)\geq\operatorname{Im}a_i-|t|\,|b_i-a_i|>0
\]
holds in every nonconstant coordinate.  
The constant coordinates remain in $\mathbb H$ as well.
Hence, by stability, each $g_j(t)$ has no zeros in the disk $|t|<\eta$, which contradicts \cref{st:root-persistence} applied to the root $0$ of $g(t)$.
\end{proof}

\subsection{Elementary symmetric polynomials and dominance}

We begin by introducing \emph{elementary symmetric polynomials} and related definitions.
For $0\leq k\leq q$, let
\[
e_k(x_1,\ldots,x_q)=
\sum_{1\leq i_1<\cdots<i_k\leq q}x_{i_1}\cdots x_{i_k},
\qquad e_0=1.
\]
Let $\CC[x_1,\ldots,x_q]^{\mathfrak S_q}$ be the ring of polynomials invariant under permutations of the variables, and set
\[
\mathcal S_{q,\leq R}
=\{F\in\CC[x_1,\ldots,x_q]^{\mathfrak S_q}:\deg F\leq R\}.
\]
Its real form consists of the polynomials with real coefficients.

A partition of length at most $q$ is written
$\lambda=(\lambda_1,\ldots,\lambda_q)$ with
$\lambda_1\geq\cdots\geq\lambda_q\geq0$, including trailing zeros.
Its size and diagram are
\[
|\lambda|=\sum_i\lambda_i,
\qquad
[\lambda]=\{(i,j):1\leq i\leq q,\ 1\leq j\leq\lambda_i\}.
\]
The conjugate partition has parts
$\lambda'_j=\#\{i:\lambda_i\geq j\}$.
For partitions of equal size, the dominance relation
$\mu\preceq\lambda$ means
\begin{equation}\label{st:dominance}
\sum_{i=1}^k\mu_i\leq\sum_{i=1}^k\lambda_i
\qquad(1\leq k\leq q).
\end{equation}
Write $\mu\prec\lambda$ if also $\mu\ne\lambda$.
The \emph{monomial symmetric polynomial} $m_\lambda$ is the sum of the
distinct monomials whose exponent vectors are permutations of
$\lambda$.  These polynomials form a basis in each degree, since symmetry makes coefficients constant on each monomial orbit.
Define
\begin{equation}\label{st:elementary-product}
E_\lambda=\prod_{j=1}^{\lambda_1}e_{\lambda'_j},
\qquad
\mathcal V_\lambda
=\operatorname{span}_{\CC}
 \{m_\mu:\mu\preceq\lambda\}.
\end{equation}
The partitions $\mu\preceq\lambda$ form the dominance downset of
$\lambda$.  The space $\mathcal V_\lambda$ is homogeneous of
degree $|\lambda|$; for the zero partition the empty product
$E_\lambda$ is $1$.

\begin{lemma}\label{st:elementary-unitriangular}
The polynomial $E_\lambda$ is real stable, and
\begin{equation}\label{st:elementary-triangular}
E_\lambda=m_\lambda+
 \sum_{\mu\prec\lambda}a_{\lambda\mu}m_\mu,
\qquad a_{\lambda\mu}\in\NN.
\end{equation}
The polynomials $E_\mu$ with $\mu\preceq\lambda$ therefore form
a basis of $\mathcal V_\lambda$.
\end{lemma}
\begin{proof}
Fix $x_i\in\mathbb H$.  The roots of $\prod_i(u+x_i)$ lie below the real axis.  If $p(u)$ is any nonconstant polynomial with roots $\rho_j$ in that half-plane, then $p'(u)/p(u)=\sum_j1/(u-\rho_j)$ has negative imaginary part for $\operatorname{Im}u\geq0$.  
We have imaginary part $\operatorname{Im}(u-\rho_j)>0$. Then we have $\operatorname{Im}\frac{1}{(u-\rho_j)}<0$.
Thus $p'(u)$ has no root in the closed upper half-plane.  Repeated differentiation shows that
\[
\left.\frac{\partial^{q-k}}{\partial u^{q-k}}
\prod_{i=1}^q(u+x_i)\right|_{u=0}
=(q-k)!e_k(x)
\]
is nonzero.  This includes $k=0$, when the derivative is constant.
Hence each $e_k$, and therefore $E_\lambda$, is real stable.

A term of $E_\lambda$ is specified by subsets
$S_j\subseteq\{1,\ldots,q\}$ of sizes $\lambda'_j$; the
exponent of $x_i$ counts the sets containing $i$.  
For any set $I$ of $k$ indices, we consider any $k$ variables.
For each choice of a subset $S_j$ of size $\lambda'_j$, at most $\min(k,\lambda'_j)$ of the $k$ variables can be selected.
The sum of its exponents is bounded by
\[\sum_j\min(k,\lambda'_j)=\sum_{i=1}^k\lambda_i.
\]
This means that, for any monomial arising from the expansion, if we take the largest $k$ exponents and sum them (i.e., arrange them in decreasing order as the first $k$ parts of a partition $\mu$), we necessarily have
\[\sum_{i=1}^k \mu_i \le \sum_{i=1}^k \lambda_i.
\]
This is precisely the definition of the dominance order \(\mu \le \lambda\).
The coefficients are nonnegative integers because they count the subset choices.

For the monomial $x_1^{\lambda_1}\cdots x_q^{\lambda_q}$,
equality holds in the bound for every initial set
$\{1,\ldots,k\}$.  
Each subset must attain its individual bound.
For $\lambda'_j=h$, equality at $k=h$ forces $S_j=\{1,\ldots,h\}$.  
These unique choices produce the required monomial. Thus, its coefficient is $1$.  
By symmetry, the coefficient of $m_{\lambda}$ is $1$, proving \eqref{st:elementary-triangular}.

There are only finitely many partitions of $|\lambda|$.
The resulting triangular matrix has diagonal entries one and is invertible on each finite dominance downset.  
This proves the basis assertion.  
It also proves that every symmetric polynomial has a unique expression in $e_1,\ldots,e_q$: 
in each degree the products $E_\mu$ are precisely the monomials $e_1^{b_1}\cdots e_q^{b_q}$,
where the conjugate partition of $\mu$ has $b_k$ parts equal to $k$.
\end{proof}

\subsection{A univariate coefficient multiplier}

For a linear map $T$ on a finite-dimensional polynomial space, write
\begin{equation}\label{st:matrix-exponential}
 e^{tT}=\sum_{j=0}^{\infty}\frac{t^jT^j}{j!}.
\end{equation}
The reader should be careful not to confuse the $e$ here with the intermediate quantity $e$ introduced in \eqref{eq:parameters}.
With any induced operator norm, the series converges absolutely and satisfies $\|e^{tT}\|\leq e^{|t|\|T\|}$.  
Multiplication of the series yields
\[
e^{sT}e^{tT}
=\sum_{k=0}^{\infty}T^k\sum_{j=0}^k\frac{s^j t^{k-j}}{j!(k-j)!}
=\sum_{k=0}^{\infty}\frac{(s+t)^kT^k}{k!}
=e^{(s+t)T}.
\]
Hence $(e^{tT})^{-1}=e^{-tT}$.  
Termwise differentiation is valid on bounded intervals. This yields $\frac{\mathrm d}{\mathrm dt}e^{tT}=Te^{tT}$.
For commuting operators, the same calculation and the binomial formula yield $e^{S+T}=e^Se^T$.

On $\CC[z]_{\leq R}$ set
\[
\mathcal E=z\frac{\mathrm d}{\mathrm dz}, \qquad \mathcal A_1=z^2\frac{\mathrm d^2}{\mathrm dz^2}.
\]
For $t\geq0$, define
\begin{equation}\label{st:multiplier-definition}
M_t(z^k)=e^{-t k(k-1)}z^k\qquad(k\geq0).
\end{equation}
Since $\mathcal A_1z^k=k(k-1)z^k$, this is $M_t=\exp(-t\mathcal A_1)$.  
Its action is independent of the chosen degree bound $R$.

\begin{lemma}\label{st:first-order}
Let $f(z)\in\CC[z]$ be nonzero of degree $r$, with all zeros in $\mathbb H$.  
If $r>0$ and $a>-1/r$ is real, then $f(z)+azf'(z)$ has degree $r$ and all zeros in $\mathbb H$.  
For $r=0$ the conclusion holds for every real $a$.
\end{lemma}
\begin{proof}
For $r>0$, let $\rho_1,\ldots,\rho_r$ be the roots, counted with multiplicities.  
For real $x$, we have 
\begin{equation}\label{st:logarithmic-derivative}
\operatorname{Im}\frac{f'(x)}{f(x)}
=\operatorname{Im}\sum_{j=1}^r\frac1{x-\rho_j}
=\sum_{j=1}^r\frac{\operatorname{Im}\rho_j}{|x-\rho_j|^2}>0.
\end{equation}
Thus $f(x)+axf'(x)=0$ is impossible when $ax\ne0$, since it would make $f'(x)/f(x)=-1/(ax)$ real.  
When $ax=0$, the value is $f(x)\ne0$.

The same argument applies along the path $f(z)+\tau azf'(z)$, $0\leq\tau\leq1$.  
Its leading coefficient is multiplied by $1+\tau ar\geq\min(1,1+ar)>0$. Thus monic normalization yields
a continuous path of degree-$r$ polynomials with no real roots.
Let $J$ be the set of parameters $\tau$ for which all roots lie in $\mathbb H$.  
It contains $0$ and is relatively open by root continuity.  
It is relatively closed because a limiting root belongs to $\overline{\mathbb H}$ and cannot be real.
Since $[0,1]$ is connected, it follows that $J=[0,1]$.
This proves the assertion.
\end{proof}

\begin{proposition}\label{st:multiplier-preservation}
For every real $t\geq0$, the operator $M_t$ has the following properties.
\begin{enumerate}
\item[(i)] It preserves the property that all univariate zeros belong to $\mathbb H$, and also the property that they belong to $\overline{\mathbb H}$.
\item[(ii)] Both assertions hold for the lower half-plane.
\item[(iii)] Applied in any one variable of a multivariate polynomial, it preserves stability and real stability.
\end{enumerate}
\end{proposition}
\begin{proof}
The assertions are immediate for $t=0$ and for nonzero constants.
For $\varepsilon>0$, set
\[
c=\frac{1+\sqrt{1+4/\varepsilon}}2,
\qquad b=c-1.
\]
Then $b,c>0$, $c-b=1$, and $bc=1/\varepsilon$.  Since
$\mathcal E^2-\mathcal E=\mathcal A_1$, we have
\begin{equation}\label{st:euler-factorization}
I-\varepsilon\mathcal A_1
=\left(I-\frac{\mathcal E}{c}\right)
 \left(I+\frac{\mathcal E}{b}\right).
\end{equation}
Indeed, the coefficients of $\mathcal E$ and $\mathcal E^2$
on the right are $1/b-1/c=\varepsilon$ and
$-1/(bc)=-\varepsilon$.

Fix a degree bound $R$ and a nonconstant polynomial $f$ of degree $r\leq R$.

For sufficiently small $\varepsilon>0$, we have $c>R$, so that $-1/c>-1/r$ and $1/b>0$ for every integer $1\le r\le R$.
Hence \cref{st:first-order} applies to both factors, with $a=-1/c$ and $a=1/b$, respectively.
Consequently, if all roots of $f$ belong to $\mathbb H$, the same is true of
\[
\left(I-\frac{t}{v}\mathcal A_1\right)^v f
\]
for all sufficiently large integers $v$.
Its multiplier on $z^k$ is $\bigl(1-tk(k-1)/v\bigr)^v\to e^{-tk(k-1)}$ since $\lim_{n\rightarrow \infty} \left(1+\frac{x}{n}\right)^n=e^{x}$.
The limit is $M_tf$, of the same degree $r$ as $f$, with leading coefficient multiplied by $e^{-tr(r-1)}\ne0$.
\cref{st:root-continuity} therefore puts its roots in $\overline{\mathbb H}$.

If the roots of $f$ are only known to belong to $\overline{\mathbb H}$, 
we apply this conclusion to $f(z-\mathrm i\delta)$. Let $0<\delta\rightarrow 0$.
Thus $M_t$ preserves the closed-root condition.
To recover the strict condition, we consider its monic normalization on degree-$r$ coefficient space:
\begin{align*}
\Phi_t\left(z^r+\sum_{k=0}^{r-1}a_kz^k\right)
=e^{tr(r-1)}M_t\left(z^r+\sum_{k=0}^{r-1}a_kz^k\right)=z^r+\sum_{k=0}^{r-1}e^{t[r(r-1)-k(k-1)]}a_kz^k.
\end{align*}
This is a homeomorphism, since all coefficient multipliers are nonzero, and it maps $\mathcal C_r^+$ into itself.  
The image of an open neighborhood contained in $\mathcal C_r^+$ is again an
open subset of $\mathcal C_r^+$.  
Hence $\Phi_t$ preserves its interior, which is exactly the strict upper-root condition by \cref{st:root-continuity}.  
This proves \textup{(i)}.
Since $M_t$ commutes with complex conjugation of coefficients, \textup{(ii)} follows directly from \textup{(i)}.

For \textup{(iii)}, specialize all other variables to arbitrary points of $\mathbb H$.  
Stability gives a nonzero univariate polynomial whose roots lie in the closed lower half-plane.
Its image under $M_t$ retains that property and remains nonzero, since its degree is preserved and its leading coefficient remains
nonzero.
It is therefore zero-free in $\mathbb H$.  
Since the specialization was arbitrary, the multivariate image is stable. 
Real coefficients are preserved by \eqref{st:multiplier-definition}.
\end{proof}

\subsection{A derivation on symmetric polynomials}

On $\CC[x_1,\ldots,x_q]^{\mathfrak S_q}$, write
$\partial_i=\partial/\partial x_i$ and define
\begin{equation}\label{st:operators}
\mathcal A_q=\sum_{i=1}^q x_i^2\partial_i^2,\qquad
\mathcal B_q=2\sum_{1\leq i<j\leq q}\frac{x_i^2\partial_i-x_j^2\partial_j}{x_i-x_j},\qquad
\mathcal D_q=\mathcal B_q-\mathcal A_q.
\end{equation}
The operator $\mathcal D_q$ is twice the Laplace--Beltrami operator
at parameter $-1$ in Stanley's normalization
\cite[p.~84, (11)]{StanleyJack}.  We prove the properties of this
specialization needed here directly.

For symmetric $F$, the numerator $x_i^2\partial_iF-x_j^2\partial_jF$ vanishes at $x_i=x_j$:
differentiating symmetry gives equality of the two partial derivatives on that diagonal.  
The numerator is therefore divisible by $x_i-x_j$.  
Thus $\mathcal B_qF$ is a polynomial, and permuting variables merely permutes its unordered pair terms.
Both $\mathcal A_q$ and $\mathcal B_q$ preserve symmetry, homogeneous degree, and real coefficients.  
In particular, all three operators preserve $\mathcal S_{q,\leq R}$.
The ordinary product rule yields
\begin{equation}\label{st:derivation-rule}
\mathcal B_q(FG)=(\mathcal B_qF)G+F(\mathcal B_qG)
\end{equation}
for symmetric $F,G$: first as a rational function identity, then as a polynomial identity by divisibility.  
Hence $\mathcal B_q$ is a derivation of the symmetric polynomial ring.

\begin{lemma}\label{st:elementary-eigenvalues}
For $0\leq k\leq q$, we have 
\begin{equation}\label{st:elementary-eigenvalue-formula}
\mathcal B_q e_k=k(2q-k-1)e_k.
\end{equation}
\end{lemma}
\begin{proof}
For $q=1$ both sides are zero.  
For $q\geq2$ and $i<j$, let $e_r^{(i,j)}$ be the elementary symmetric polynomial in the other
$q-2$ variables, with value zero outside $0\leq r\leq q-2$.
Then
\[
e_k=e_k^{(i,j)}+(x_i+x_j)e_{k-1}^{(i,j)}+x_ix_je_{k-2}^{(i,j)}.
\]
The pair operator $2\frac{x_i^2\partial_i-x_j^2\partial_j}{x_i-x_j}$ vanishes on $e_k^{(i,j)}$, while
\begin{align*}
2\frac{(x_i^2\partial_i-x_j^2\partial_j)(x_i+x_j)}{x_i-x_j}&=2\frac{x_i^2-x_j^2}{x_i-x_j}=2(x_i+x_j),\\
2\frac{(x_i^2\partial_i-x_j^2\partial_j)(x_ix_j)}{x_i-x_j}&=2\frac{x_i^2x_j-x_j^2x_i}{x_i-x_j}=2x_ix_j.
\end{align*}
Consequently, the pair contributes a multiplier $2$ to each squarefree monomial whose support meets $\{i,j\}$.
A fixed $k$-element support meets $\binom{k}{2}$ pairs in two elements and $k(q-k)$ pairs in one.  
Its total multiplier is
\[
2\left\{\binom{k}{2}+k(q-k)\right\}=k(2q-k-1).
\]
This includes $k=0$ and proves the formula.
\end{proof}

By \cref{st:elementary-unitriangular}, write each symmetric polynomial uniquely as $F=G(e_1,\ldots,e_q)$.  
The derivation rule and \eqref{st:elementary-eigenvalue-formula} yield
\[
\mathcal B_q(e_1^{b_1}\cdots e_q^{b_q})=\left(\sum_{k=1}^q b_k k(2q-k-1)\right)e_1^{b_1}\cdots e_q^{b_q}.
\]
The exponential therefore multiplies this elementary monomial by
\[\exp\left(t\sum_{k=1}^q b_k k(2q-k-1)\right)
=\prod_{k=1}^q\left(e^{tk(2q-k-1)}\right)^{b_k}.
\]
By linearity, we have
\begin{equation}\label{st:derivation-exponential}
\bigl(e^{t\mathcal B_q}F\bigr)(x)=G\left(e^{t(2q-2)}e_1(x),e^{2t(2q-3)}e_2(x),\ldots,e^{qt(q-1)}e_q(x)\right).
\end{equation}
Here the indexed expression gives the $q$ arguments of $G$.
The formula holds on every $\mathcal S_{q,\leq R}$ containing $F$.

\begin{proposition}\label{st:derivation-preserves}
For $t\geq0$ and $R\in\NN$, the operator $e^{t\mathcal B_q}$ on $\mathcal S_{q,\leq R}$ preserves stability and real stability.
\end{proposition}
\begin{proof}
For $x=(x_1,\ldots,x_q)\in\mathbb H^q$, form
\[
p_x(u)=\prod_{i=1}^q(u-x_i)=\sum_{k=0}^q(-1)^k e_k(x)u^{q-k}.
\]
Set
\begin{equation}\label{st:root-lift}
p_y(u)=e^{q(q-1)t}M_t p_x(u).
\end{equation}
The leading coefficient is $e^{q(q-1)t}e^{-q(q-1)t}=1$.
By \cref{st:multiplier-preservation}, all roots $y_1,\ldots,y_q$ of $p_y$ lie in $\mathbb H$.
Comparison of the coefficients of $u^{q-k}$ yields
\begin{align}\label{st:root-lift-coefficients}
e_k(y)=e^{[q(q-1)-(q-k)(q-k-1)]t}e_k(x)=e^{k(2q-k-1)t}e_k(x).
\end{align}
Together with \eqref{st:derivation-exponential}, this implies
\[
\bigl(e^{t\mathcal B_q}F\bigr)(x)=G(e_1(y),\ldots,e_q(y))=F(y_1,\ldots,y_q).
\]
By symmetry, the right-hand side is independent of the ordering of the roots $y_1,\dots,y_q$, 
and it is nonzero whenever $F$ is stable.
Thus $e^{t\mathcal B_q}F$ is stable.  
The operator also preserves real coefficients.
\end{proof}

\begin{theorem}\label{st:evolution-preserves}
For $t\geq0$ and $R\in\NN$, the operator $e^{t\mathcal D_q}$
on $\mathcal S_{q,\leq R}$ preserves stability and real stability.
\end{theorem}
\begin{proof}
Since the operators $x_i^2\partial_i^2$ commute, $e^{-t\mathcal A_q}$ is the product of the coordinate operators $M_t$.  
Each preserves stability by \cref{st:multiplier-preservation}.  
Their product preserves symmetry because permutations merely permute the commuting factors. 
\cref{st:derivation-preserves} yields both properties for $e^{t\mathcal B_q}$.

Note that \(\mathcal{A}_q\) and \(\mathcal{B}_q\) do not commute in general (\(\mathcal{A}_q\mathcal{B}_q \neq \mathcal{B}_q\mathcal{A}_q\)). 
Hence the exponential cannot be split directly, i.e., \(e^{t(\mathcal{B}_q-\mathcal{A}_q)} \neq e^{t\mathcal{B}_q}e^{-t\mathcal{A}_q}\) in general.
Therefore, we use the following finite dimensional product formula.

On the finite dimensional space $\mathcal S_{q,\leq R}$, the product formula is
\begin{equation}\label{st:trotter}
e^{t(\mathcal B_q-\mathcal A_q)}=\lim_{v\longrightarrow\infty}\left(e^{(t/v)\mathcal B_q}e^{-(t/v)\mathcal A_q}\right)^v.
\end{equation}
This is the finite dimensional case of Trotter's formula \cite{Trotter}; the following estimate supplies a direct proof.
For $t>0$, set $h=t/v$ and
\[
U_h=e^{h\mathcal B_q}e^{-h\mathcal A_q},\qquad
V_h=e^{h(\mathcal B_q-\mathcal A_q)}.
\]
For a fixed matrix $T$ and $0\leq h\leq1$, we have 
\begin{equation}\label{st:exponential-remainder}
\|e^{hT}-I-hT\|\leq\sum_{j=2}^{\infty}\frac{h^j\|T\|^j}{j!}\leq\frac{h^2\|T\|^2}{2}e^{h\|T\|}.
\end{equation}
The last inequality uses $j!\geq2(j-2)!$.
By expanding the two factors of $U_h$ and the exponential $V_h$, we show that both have linear part
$I+h(\mathcal B_q-\mathcal A_q)$, with remainders bounded by a constant times $h^2$.  
Thus $\|U_h-V_h\|\leq Ch^2$.
For $C_1=\|\mathcal B_q\|+\|\mathcal A_q\|$, we also have
\[
\|U_h\|\leq e^{h\|\mathcal B_q\|}e^{h\|\mathcal A_q\|}=e^{C_1h},\qquad
\|V_h\|\leq e^{h\|\mathcal B_q-\mathcal A_q\|}\leq e^{C_1h}.
\]
The telescoping identity
\[
U_h^v-V_h^v=\sum_{j=0}^{v-1}U_h^{v-1-j}(U_h-V_h)V_h^j
\]
requires no commutativity and yields
\[
\|U_h^v-V_h^v\|\leq v C h^2 e^{C_1h(v-1)}\leq \frac{Ct^2e^{C_1t}}{v}\longrightarrow0.
\]
Since $V_h^v=e^{t\mathcal D_q}$, this proves \eqref{st:trotter}; the case $t=0$ is immediate.

Every approximating product preserves stability and symmetry.
Its coefficientwise limit is stable or zero by \cref{st:limit-closure}.  
It is nonzero on a nonzero input because $e^{t\mathcal D_q}$ is invertible, with inverse $e^{-t\mathcal D_q}$.  
All operators preserve real coefficients.
\end{proof}

\subsection{The stable eigenfunction}

\begin{lemma}\label{st:operator-triangular}
For every partition $\nu$ of length at most $q$, we have 
\begin{equation}\label{st:triangular-operator}
\mathcal D_q m_\nu=\varepsilon_\nu m_\nu+\sum_{\mu\prec\nu}c_{\nu\mu}m_\mu,
\end{equation}
where $c_{\nu\mu}\in\RR$ and
\begin{align}
\varepsilon_\nu&=\sum_{i=1}^q\nu_i\bigl(2(q-i)-\nu_i+1\bigr)\label{st:eigenvalue}\\
&=2\sum_{(i,j)\in[\nu]}(q+1-i-j).\label{st:eigenvalue-diagram}
\end{align}
In particular, $\mathcal V_\lambda$ is invariant under $\mathcal D_q$.
\end{lemma}
\begin{proof}
By the derivation rule and \cref{st:elementary-eigenvalues}, we obtain 
\[
\mathcal B_qE_\nu=\beta_\nu E_\nu,
\qquad
\beta_\nu=\sum_{j=1}^{\nu_1}\nu'_j(2q-\nu'_j-1).
\]
By counting by rows and columns, we have 
\begin{equation}\label{st:beta-row-column}
\beta_\nu=\sum_{j=1}^{\nu_1}\sum_{i=1}^{\nu'_j}2(q-i)
=2\sum_{i=1}^q(q-i)\nu_i.
\end{equation}
Invert the unitriangular expansion of \cref{st:elementary-unitriangular} on the downset of $\nu$:
for real coefficients $b_{\nu\mu}$,
\[
m_\nu=E_\nu+\sum_{\mu\prec\nu}b_{\nu\mu}E_\mu.
\]
Hence, 
\[
\mathcal B_qm_\nu
=\beta_\nu E_\nu+
 \sum_{\mu\prec\nu}b_{\nu\mu}\beta_\mu E_\mu.
\]
Re-expanding the elementary products according to
$$E_\mu=m_\mu+\sum_{\eta\prec\mu}a_{\mu\eta}m_\eta,$$
the first term contributes $\beta_\nu m_\nu$ and lower terms.  
Every remaining term is supported on $\eta\preceq\mu\prec\nu$.  
Thus, $\mathcal B_q$ is triangular with diagonal entry $\beta_\nu$.
Moreover,
\[\mathcal A_qm_\nu=\left(\sum_{i=1}^q\nu_i(\nu_i-1)\right)m_\nu,
\]
since each monomial in the orbit has the same exponent multiset.  

By subtracting $\mathcal{A}_q m_\nu$ from $\mathcal{B}_q m_\nu$, 
we obtain \eqref{st:triangular-operator} and \eqref{st:eigenvalue}. 
Finally, for each $1\le i\le q$,
\[
2\sum_{j=1}^{\nu_i}(q+1-i-j)
=2\nu_i(q+1-i)-\nu_i(\nu_i+1)
=\nu_i\bigl(2(q-i)-\nu_i+1\bigr),
\]
and summing over $1\le i\le q$ establishes \eqref{st:eigenvalue-diagram}.
\end{proof}

\begin{lemma}\label{st:spectral-gap}
Suppose that $\lambda$ has length at most $q$ and satisfies
\begin{equation}\label{st:adjacent-drop}
\lambda_i-\lambda_{i+1}\leq1\qquad(1\leq i<q).
\end{equation}
Then $\varepsilon_\lambda>\varepsilon_\mu$ for every
$\mu\prec\lambda$ of length at most $q$.
\end{lemma}
\begin{proof}
We refer to the cells of $[\lambda]\setminus[\mu]$ as removed and to those of $[\mu]\setminus[\lambda]$ as added. 
No row contains cells of both types.
Dominance says that each initial collection of rows contains at least as many removed cells as added cells.  
Scanning the rows therefore pairs each added cell with an unpaired removed cell in an earlier row: 
an added row has no removed cells, and the dominance inequality after the entire row guarantees enough earlier removed cells.  
Equal sizes ensure that every removed cell is paired.

For a paired removed cell $(i,j)$ and added cell $(k,\ell)$, where $k>i$, we have
\[
j\leq\lambda_i,\qquad \ell\geq\lambda_k+1.
\]
By the condition \eqref{st:adjacent-drop}, we have $\lambda_i-\lambda_k\leq k-i$, whence
\begin{equation}\label{st:cell-level-increase}
(k+\ell)-(i+j)\geq k-i+\lambda_k+1-\lambda_i\geq1.
\end{equation}
Since the common cells in $[\lambda]\cap[\mu]$ contribute equally to both sums, 
summing the preceding inequality over the nonempty set of pairs yields
\[
\sum_{(i,j)\in[\mu]}(i+j)>\sum_{(i,j)\in[\lambda]}(i+j).
\]
By \eqref{st:eigenvalue-diagram} and equality of sizes, we obtain the required eigenvalue inequality.  
In fact, if $r=|[\lambda]\setminus[\mu]|$, the same pairing proves
\begin{equation}\label{st:gap-lower-bound}
\varepsilon_\lambda-\varepsilon_\mu\geq2r>0.
\end{equation}
This completes the proof.
\end{proof}

\begin{theorem}\label{st:stable-eigenfunction}
Let $q\geq1$ and let $\lambda$ have length at most $q$ and satisfy
\eqref{st:adjacent-drop}.  There is a unique polynomial
$P_\lambda\in\mathcal V_\lambda$ of the form
\begin{equation}\label{st:monic-eigenfunction-form}
P_\lambda=m_\lambda+\sum_{\mu\prec\lambda}p_{\lambda\mu}m_\mu
\end{equation}
satisfying
\begin{equation}\label{st:monic-eigenfunction-equation}
\mathcal D_qP_\lambda=\varepsilon_\lambda P_\lambda.
\end{equation}
It has real coefficients and is real stable.  Consequently, every
polynomial in $\mathcal V_\lambda$ with a nonzero $m_\lambda$
coefficient and satisfying \eqref{st:monic-eigenfunction-equation}
is stable.
\end{theorem}
\begin{proof}
We may choose a total order on the finite dominance downset of $\lambda$ such that $\nu$ precedes $\mu$ whenever $\mu < \nu$. 
Set $p_{\lambda\lambda}=1$.
By \cref{st:operator-triangular}, we have
$$\mathcal D_q m_\nu=\varepsilon_\nu m_\nu+\sum_{\mu\prec\nu}c_{\nu\mu}m_\mu$$
According to
$$\mathcal D_q\left(m_\lambda+\sum_{\mu\prec\lambda}p_{\lambda\mu}m_\mu\right)=\varepsilon_\lambda \left(m_\lambda+\sum_{\mu\prec\lambda}p_{\lambda\mu}m_\mu\right),$$
the $m_\mu$ coefficient of the eigenvalue equation is
\[\varepsilon_\mu p_{\lambda\mu}+\sum_{\mu\prec\nu\preceq\lambda}c_{\nu\mu}p_{\lambda\nu}=\varepsilon_\lambda p_{\lambda\mu}.
\]
For $\mu=\lambda$ this is automatic; otherwise it gives
\begin{equation}\label{st:eigenfunction-recursion}
p_{\lambda\mu}
=\frac{\displaystyle\sum_{\mu\prec\nu\preceq\lambda}c_{\nu\mu}p_{\lambda\nu}}{\varepsilon_\lambda-\varepsilon_\mu}.
\end{equation}
All coefficients on the right occur earlier in the order, and every
denominator is positive by \cref{st:spectral-gap}.
Thus the recursion gives existence, uniqueness, and real coefficients.

To prove stability, set
\begin{equation}\label{st:normalized-evolution}
P_t=e^{-t\varepsilon_\lambda}
       e^{t\mathcal D_q}E_\lambda,
\qquad t\geq0.
\end{equation}
By \cref{st:elementary-unitriangular} and \cref{st:evolution-preserves}, each $P_t$ is real stable.
The invariance of $\mathcal V_\lambda$ keeps $P_t$ in that space,
and differentiation yields
\begin{equation}\label{st:normalized-differential-equation}
P_t'=(\mathcal D_q-\varepsilon_\lambda I)P_t.
\end{equation}
By the triangularity of $\mathcal{D}_q$, the coefficient of $m_\lambda$ in $P_t$ remains equal to its initial value $1$.
Write
\[
P_t=\sum_{\mu\preceq\lambda}u_\mu(t)m_\mu,\qquad u_\lambda(t)=1.
\]
For $\mu\prec\lambda$, by comparing coefficients according to \eqref{st:normalized-differential-equation}, we obtain
\begin{equation}\label{st:coefficient-evolution}
u_\mu'(t)
=-(\varepsilon_\lambda-\varepsilon_\mu)u_\mu(t)+\sum_{\mu\prec\nu\preceq\lambda}c_{\nu\mu}u_\nu(t).
\end{equation}
Here $c_{\nu\mu}=0$ for absent terms.

We prove coefficient convergence in the same total order.  
Suppose that $u_{\nu}(t)$ converges for every partition $\nu$ preceding $\mu$ in the chosen order.
Then $v_\mu(t)=\sum_{\mu\prec\nu\preceq\lambda}c_{\nu\mu}u_\nu(t)$ is bounded and has a limit $v_\mu(\infty)$.
Set $\delta_\mu=\varepsilon_\lambda-\varepsilon_\mu>0$.
According to \eqref{st:coefficient-evolution}, we have
$$u'_\mu(t) + \delta_\mu u_\mu(t) = v_\mu(t)$$
Thus, we have $e^{\delta_\mu t} u'_\mu(t) + e^{\delta_\mu t}\delta_\mu u_\mu(t) =e^{\delta_\mu t} v_\mu(t)$.
Therefore, we obtain
$$(e^{\delta_\mu t} u_\mu(t))'=e^{\delta_\mu t} v_\mu(t).$$
Integrating from $0$ to $t$ and then dividing by $e^{\delta_\mu t}$, we obtain 
\begin{equation}\label{st:variation-of-constants}
u_\mu(t)=e^{-\delta_\mu t}u_\mu(0)
 +\int_0^t e^{-\delta_\mu(t-s)}v_\mu(s)\,\mathrm ds.
\end{equation}
It follows that
\begin{align*}
u_\mu(t)-\frac{v_\mu(\infty)}{\delta_\mu}=e^{-\delta_\mu t}\left(u_\mu(0)-\frac{v_\mu(\infty)}{\delta_\mu}\right)+\int_0^t e^{-\delta_\mu(t-s)}(v_\mu(s)-v_\mu(\infty))\,\mathrm ds.
\end{align*}
For $T>0$, set $K_T=\sup_{0\leq s\leq T}|v_\mu(s)-v_\mu(\infty)|$.
When $t\geq T$, the part of the last integral over $[0,T]$ is bounded in modulus by $TK_Te^{-\delta_\mu(t-T)}$, and the
remaining part by
\[
\sup_{s\geq T}|v_\mu(s)-v_\mu(\infty)|\int_T^t e^{-\delta_\mu(t-s)}\,\mathrm ds
\leq\frac1{\delta_\mu} \sup_{s\geq T}|v_\mu(s)-v_\mu(\infty)|.
\]
First let $t\to\infty$ and then $T\to\infty$.
In this iterated limit, both bounds tend to $0$. 
Consequently, $u_\mu(t)\to v_\mu(\infty)/\delta_\mu$ as $t\to\infty$, which completes the induction.
Thus $u_\mu(t)$ converges. This proves convergence of all coefficients.

Writing $\ell_\mu=\lim_{t\to\infty}u_\mu(t)$, we obtain $\ell_\lambda=1$ and
\begin{equation}\label{st:limit-recursion}
\ell_\mu
=\frac{v_\mu(\infty)}{\delta_\mu}
=\frac{\displaystyle\sum_{\mu\prec\nu\preceq\lambda}c_{\nu\mu}\ell_\nu}{\varepsilon_\lambda-\varepsilon_\mu}
\qquad(\mu\prec\lambda).
\end{equation}
This coincides with the recursion \eqref{st:eigenfunction-recursion}, and the limit equals $P_\lambda$ by uniqueness.
The limit is nonzero, since its $m_\lambda$ coefficient is $1$.
By applying \cref{st:limit-closure} to the sequence $P_1, P_2, \ldots$, we shows that $P_\lambda$ is real stable.

Finally, dividing any polynomial as in the last assertion by its nonzero coefficient of $m_\lambda$ yields $P_\lambda$ by uniqueness. Since a nonzero scalar multiple of $P_\lambda$ has the same zero set, every such polynomial is stable.
\end{proof}

\section{A stable polynomial associated with the constant term}
\label{ct:section}

To prove real-rootedness, we realize the polynomial in \eqref{eq:positive-ct} as a specialization of a symmetric stable
polynomial.  Its constant-term definition yields the dominance support, leading coefficient, and eigenvalue equation required by \cref{st:stable-eigenfunction}.

Throughout this section, $m\geq1$ and $q\geq\max(1,2m-2)$ are integers.  
An \emph{alphabet} $b=(b_1,\ldots,b_q)$ means simply a list of commuting, algebraically independent indeterminates.  
Initially all calculations take place over $\mathbb Q$.  
Rational functions in $b,z_1,\ldots,z_m$ are expanded in the fixed field
\[\mathbb Q(b_1,\ldots,b_q)((z_1))\cdots((z_m)).
\]
Thus $z_m$ is the outermost Laurent series variable. 
For $i<j$, the inverse factor $(z_i-z_j)^{-1}$ is expanded in nonnegative powers of $z_j/z_i$.  
In this section $\CT_z$ always denotes successive extraction of the coefficient of $z_m^0$, then $z_{m-1}^0$, and so on, ending with
$z_1^0$.  
This is the map used in \eqref{eq:T-definition}; in particular,
\[
 T_m(f)=\CT_z\bigl(\Delta_m(z)^{-2}f\bigr), \qquad \Delta_m(z)=\prod_{1\leq i<j\leq m}(z_i-z_j).
\]
The finite expansion below shows that $F_{m,q}$ belongs to $\mathbb Q[b_1,\ldots,b_q]$.  
The coefficient field permits rational identities throughout the proof.  
For stability, we regard these polynomials as elements of $\mathbb R[b_1,\ldots,b_q]$.

For the support calculation, set
\begin{equation}\label{ct:partitions}
 a_i=2(m-i)\quad(1\leq i\leq m),\qquad \lambda=(m-1,m-1,m-2,m-2,\ldots,1,1).
\end{equation}
We append trailing zeros to $\lambda$ so that it has length $q$. 
It is conjugate to the partition $a=(2m-2,2m-4,\ldots,2,0)$, since
\[
 \lambda_r=\#\{i:a_i\geq r\}=m-\lceil r/2\rceil\quad(1\leq r\leq2m-2).
\]
Both partitions have size $m(m-1)$, because $\sum_{i=1}^m2(m-i)=m(m-1)$.  
The condition on $q$ ensures that $a_i\leq q$ for every $i$ and that $\lambda$ has at most $q$
nonzero parts.  
For $m=1$, the displayed list of positive parts of $\lambda$ is empty: it denotes the zero partition.

We retain the dominance order $\preceq$ and the monomial
symmetric polynomials $m_\nu$ from \cref{st:section}.
Explicitly, for partitions $\nu,\lambda$ of the same size,
$\nu\preceq\lambda$ means
$\sum_{r=1}^k\nu_r\leq\sum_{r=1}^k\lambda_r$ for
$1\leq k\leq q$.  The notation $\nu\prec\lambda$
means additionally $\nu\ne\lambda$.  The polynomial $m_\nu(b)$
is the sum of the distinct monomials whose exponent vectors are
permutations of $\nu$.  In particular, the coefficient of
$m_\lambda$ in a symmetric polynomial is exactly the coefficient
of the single monomial $b_1^{\lambda_1}\cdots b_q^{\lambda_q}$.

Define
\begin{equation}\label{ct:universal-definition}
 K(b,z)=\prod_{r=1}^{q}\prod_{i=1}^{m}(1+b_rz_i),\qquad
 F_{m,q}(b)=\CT_z\frac{K(b,z)}{\Delta_m(z)^2}.
\end{equation}
Empty products are one.  In particular, $F_{1,q}=1$.

\subsection{The finite expansion and its monomial support}

For $0\leq d\leq q$, the elementary symmetric polynomial is
\[
 e_d(b)=\sum_{\substack{S\subseteq\{1,\ldots,q\}\\|S|=d}}
          \prod_{r\in S}b_r.
\]
Thus $e_0=1$, and we set $e_d=0$ for $d<0$ or $d>q$.
These definitions give the polynomial identity
\[
 \prod_{r=1}^q(1+b_rz_i)=\sum_{d=0}^q e_d(b)z_i^d.
\]

Let $k=(k_{ij})_{1\leq i<j\leq m}$ be a nonnegative integer
array, viewed as a flow on the edges $i\to j$ with $i<j$.
Its netflow at $i$ is
$\sum_{j>i}k_{ij}-\sum_{h<i}k_{hi}$.

\begin{lemma}\label{ct:finite-expansion}
The constant term $F_{m,q}$, computed in
$\mathbb Q(b)((z_1))\cdots((z_m))$, is a polynomial in
$\mathbb Z[b_1,\ldots,b_q]$.  More precisely, for an array $k$
as above set
\begin{equation}\label{ct:degree-vector}
 d_i=a_i+\sum_{j>i}k_{ij}-\sum_{h<i}k_{hi}.
\end{equation}
Then one has the finite expansion
\begin{equation}\label{ct:finite-formula}
 F_{m,q}(b)=
 \sum_{\substack{k_{ij}\geq0\\0\leq d_i\leq q\ (1\leq i\leq m)}}
 \left(\prod_{i<j}(k_{ij}+1)\right)\prod_{i=1}^m e_{d_i}(b).
\end{equation}
In particular, $F_{m,q}$ is symmetric, homogeneous of degree $m(m-1)$,
and has nonnegative integer coefficients.
\end{lemma}
\begin{proof}
For each $i<j$, the embedding in the fixed Laurent-series field yields
\[
 (z_i-z_j)^{-2}=z_i^{-2}\sum_{k_{ij}\geq0}(k_{ij}+1)(z_j/z_i)^{k_{ij}}.
\]
An array $k$ selects one term from each factor.  
The resulting exponent of $z_i$ (in the expansion of $\Delta_m(z)^{-2}$) is
\[
-2(m-i)-\sum_{j>i}k_{ij}+\sum_{h<i}k_{hi}=-d_i.
\]
Its coefficient is $\prod_{i<j}(k_{ij}+1)$, and the numerator must supply the coefficient $\prod_i e_{d_i}(b)$.  
This yields the summands in \eqref{ct:finite-formula}.

To prove finiteness, fix $1\leq r<m$ and sum \eqref{ct:degree-vector} over $1\leq i\leq r$. 
The contributions of all internal edges cancel, yielding the cut identity
\begin{equation}\label{ct:prefix}
 \sum_{i=1}^r d_i =\sum_{i=1}^r a_i+\sum_{i\leq r<j}k_{ij}.
\end{equation}
If $0\leq d_i\leq q$ for every $i$, we have 
\[
 0\leq\sum_{i\leq r<j}k_{ij}\leq rq-\sum_{i=1}^r a_i\leq rq.
\]
In particular, by setting $r=i$, we obtain $0\leq k_{ij}\leq iq$ for all $1\leq i<j\leq m$. 
Hence there are only finitely many contributing arrays, which proves \eqref{ct:finite-formula}. 
For $m=1$, the sum consists of the single empty array.

Every factor $e_{d_i}$ is symmetric in $b$ and has nonnegative integer coefficients; 
the scalar weight is a positive integer.
Finally, when \eqref{ct:degree-vector} is summed over $1\leq i\leq m$, each $k_{ij}$ appears once with each sign and therefore cancels, yielding
\begin{equation}\label{ct:total-degree}
 \sum_{i=1}^m d_i=\sum_{i=1}^m a_i=m(m-1).
\end{equation}
This proves homogeneity and all the remaining assertions.
\end{proof}

\begin{lemma}\label{ct:dominance}
In the polynomial ring $\mathbb Q[b_1,\ldots,b_q]$, one has
\begin{equation}\label{ct:monic-triangular}
 F_{m,q}=m_\lambda+\sum_{\nu\prec\lambda}c_\nu m_\nu,
\qquad c_\nu\in\NN.
\end{equation}
Equivalently, the decreasing rearrangement of every exponent
vector of a monomial occurring in $F_{m,q}$ is dominated by
$\lambda$, and the coefficient of $b^\lambda$ is exactly one.
\end{lemma}

\begin{proof}
We first bound the support of each summand in
\eqref{ct:finite-formula}, then compute the coefficient of $b^\lambda$.

Fix a contributing array in \eqref{ct:finite-formula}.  A monomial
in $\prod_{i=1}^m e_{d_i}(b)$ is specified by subsets
$S_i\subseteq\{1,\ldots,q\}$ with $|S_i|=d_i$; its exponent
of $b_r$ is the number of sets $S_i$ containing $r$.  If
$R\subseteq\{1,\ldots,q\}$ has size $\ell$, the sum of the
exponents at positions in $R$ is $\sum_{i=1}^m|S_i\cap R|$.  Since
$|S_i\cap R|\leq|R|=\ell$ and
$|S_i\cap R|\leq|S_i|=d_i$ separately for every $i$, it follows
that
\begin{equation}\label{ct:subset-count}
 \sum_{i=1}^m|S_i\cap R|\leq\sum_{i=1}^m\min(\ell,d_i).
\end{equation}
We claim that
\begin{equation}\label{ct:min-bound}
 \sum_{i=1}^m\min(\ell,d_i)\leq\sum_{i=1}^m\min(\ell,a_i).
\end{equation}
This does not require the entries $d_i$ to be decreasing.
Indeed, fix $1\leq \ell\leq q$ and let
$p=\#\{i:a_i>\ell\}$.  Since $a$ is decreasing, these
indices form the prefix $\{1,\ldots,p\}$.  Write
$(u)_+=\max(u,0)$.  Equation~\eqref{ct:prefix} implies
\begin{align*}
\sum_{i=1}^m(d_i-\ell)_+\geq\sum_{i=1}^{p}(d_i-\ell)\geq\sum_{i=1}^{p}(a_i-\ell)=\sum_{i=1}^m(a_i-\ell)_+.
 \end{align*}
The first inequality discards nonnegative terms outside the prefix
and uses $(d_i-\ell)_+\geq d_i-\ell$ inside it.  The second follows
from \eqref{ct:prefix}, or is immediate when $p=0$; since $a_m=0$,
we have $p<m$.  
By subtracting from $\sum_{i=1}^m d_i=\sum_{i=1}^m a_i=m(m-1)$, we obtain 
$$\sum_{i=1}^m d_i-\sum_{i=1}^m(d_i-\ell)_+ \leq \sum_{i=1}^m a_i- \sum_{i=1}^m(a_i-\ell)_+$$
From $\min(\ell,u)=u-(u-\ell)_+$, we prove \eqref{ct:min-bound}.

Let $\nu$ be the decreasing rearrangement of the monomial's
exponents, and take $R$ to contain the positions of the $\ell$
largest exponents.  By \eqref{ct:subset-count} and
\eqref{ct:min-bound}, we have 
\begin{equation}\label{ct:dominance-chain}
 \sum_{r=1}^\ell\nu_r
 \leq\sum_{i=1}^m\min(\ell,d_i)
 \leq\sum_{i=1}^m\min(\ell,a_i)
 =\sum_{r=1}^\ell\lambda_r.
\end{equation}
The last equality is the conjugacy relation between $a$ and
$\lambda$.  Together with \eqref{ct:total-degree}, these inequalities
prove $\nu\preceq\lambda$.

We prove next that the coefficient of the particular monomial
$b^\lambda=b_1^{\lambda_1}\cdots b_q^{\lambda_q}$ is exactly
one.  The zero array $k_{ij}=0$, together with the choices
$S_i=\{1,\ldots,a_i\}$, produces this monomial and has weight
one. Thus, its coefficient is at least one.  
Indeed, the number of these sets containing $r$ is $\#\{i:a_i\geq r\}=\lambda_r$.

Suppose conversely that some contributing array and subsets produce $b^\lambda$.  
We may choose $R=\{1,\ldots,\ell\}$ for $1\leq \ell\leq q$.  
Then the sum of the selected exponents is $\lambda_1+\cdots+\lambda_\ell$. 
Thus, the two ends of \eqref{ct:dominance-chain} are equal for each such $\ell$.
All intermediate inequalities must therefore be equalities, and in particular
\begin{align}\label{Equation-Elll-AD}
 \sum_{i=1}^m\min(\ell,d_i)=\sum_{i=1}^m\min(\ell,a_i)\quad\text{for every integer }\ell\geq0.
\end{align}
For $\ell=0$ both sides are zero; for $\ell>q$ both are $m(m-1)$.
Note that $\min(\ell,u)-\min(\ell-1,u)$ equals 1 when $u\geq \ell$ and 0 otherwise. 
By taking the difference between $\ell$ and $\ell-1$ using equation \eqref{Equation-Elll-AD}, we obtain:
$\#\{i:d_i\geq\ell\}=\#\{i:a_i\geq\ell\}$ for every $\ell\geq1$.
Hence $d$ and $a$ have the same multiset of entries.  
Since $a$ is decreasing, the sum of any $r$ entries of $d$ cannot exceed $a_1+\cdots+a_r$.  
By comparing this fact with \eqref{ct:prefix}, we have 
\[
 \sum_{i\leq r<j}k_{ij}=0\quad(1\leq r<m).
\]
Every $k_{ij}$ occurs in the cut with $r=i$, so nonnegativity forces $k_{ij}=0$.  
Thus $d_i=a_i$ for every $i$.

There is also only one possible choice of the sets $S_i$.
Equality in \eqref{ct:subset-count} for $R=\{1,\ldots,\ell\}$ yields
\[
 \sum_{i=1}^m|S_i\cap\{1,\ldots,\ell\}|=\sum_{i=1}^m\min(\ell,a_i).
\]
Equality holds term by term.  
By taking $\ell=a_i$ when $a_i>0$, we get $S_i=\{1,\ldots,a_i\}$; for $a_i=0$, the set
$S_i$ is empty.  
Thus $b^\lambda$ has a unique contribution, of weight one.  
By combining this with symmetry and \cref{ct:finite-expansion}, we obtain \eqref{ct:monic-triangular}.
\end{proof}

\subsection{The differential equation}

For an alphabet $x=(x_1,\ldots,x_s)$, let
$\partial_{x_r}$ denote formal partial differentiation with
respect to $x_r$, treating the remaining variables as constants.
The differential operator needed here, with the same normalization
as in \cref{st:section}, is
\begin{equation}\label{ct:operator}
 \mathcal D_s
 =-\sum_{r=1}^{s}x_r^2\partial_{x_r}^2
 +2\sum_{1\leq r<t\leq s}
 \frac{x_r^2\partial_{x_r}-x_t^2\partial_{x_t}}{x_r-x_t}.
\end{equation}
For symmetric $f$, the numerator
$x_r^2\partial_{x_r}f-x_t^2\partial_{x_t}f$ is antisymmetric in
$x_r,x_t$ and hence divisible by $x_r-x_t$.  Thus
$\mathcal D_s$ preserves symmetric polynomials over any field
of characteristic zero.
We write $\mathcal D_q^{(b)}$ or $\mathcal D_m^{(z)}$ when both
alphabets occur and the variables ($b$ or $z$) being differentiated must be
specified.

\begin{lemma}\label{ct:kernel-identity}
Let $E_b=\sum_{r=1}^qb_r\partial_{b_r}$ and
$E_z=\sum_{i=1}^mz_i\partial_{z_i}$ be the derivative
operators in the two alphabets.  The following identity holds in
$\mathbb Q[b_1,\ldots,b_q,z_1,\ldots,z_m]$ for the polynomial
$K(b,z)$ defined in \eqref{ct:universal-definition}:
\begin{equation}\label{ct:kernel-intertwining}
 \mathcal D_q^{(b)}K(b,z)-\mathcal D_m^{(z)}K(b,z)=2(q-m)E_bK(b,z),\qquad E_bK(b,z)=E_zK(b,z).
\end{equation}
\end{lemma}
\begin{proof}
Work in $\mathbb Q(b_1,\ldots,b_q,z_1,\ldots,z_m)$.
Separate symmetry in the two alphabets ensures that the final
identity is polynomial.  Set
\[
 u_{ri}=\frac{b_rz_i}{1+b_rz_i},\qquad
 S=\sum_{r,i}u_{ri},\qquad
 R=\sum_{r=1}^q\sum_{i<j}u_{ri}u_{rj},\qquad
 C=\sum_{i=1}^m\sum_{r<t}u_{ri}u_{ti}.
\]
The product rule first yields
\[
 \frac{b_r\partial_{b_r}K(b,z)}{K(b,z)}=\sum_{i=1}^m u_{ri},
 \qquad
 \frac{z_i\partial_{z_i}K(b,z)}{K(b,z)}=\sum_{r=1}^q u_{ri}.
\]
By summing these formulas, we obtain $E_bK(b,z)/K(b,z)=E_zK(b,z)/K(b,z)=S$.
A second derivative of $K(b,z)$ in any one variable differentiates two distinct linear factors, each unordered pair occurring twice.  Hence
\begin{equation}\label{ct:kernel-second-derivatives}
\frac{1}{K(b,z)}\sum_{r=1}^q b_r^2\partial_{b_r}^2K(b,z)=2R,
 \qquad
 \frac{1}{K(b,z)}\sum_{i=1}^m z_i^2\partial_{z_i}^2K(b,z)=2C.
\end{equation}
Explicitly,
\[
 \frac{b_r^2\partial_{b_r}^2K(b,z)}{K(b,z)}
 =\sum_{\substack{1\leq i,j\leq m\\i\ne j}}
    \frac{b_rz_i}{1+b_rz_i}\frac{b_rz_j}{1+b_rz_j}
 =2\sum_{1\leq i<j\leq m}u_{ri}u_{rj}.
\]
The second formula follows by the same product rule with a fixed $z_i$ and two distinct indices from the $b$ alphabet.

For the first-order divided difference, the product rule yields
\[
 \frac{(b_r^2\partial_{b_r}-b_t^2\partial_{b_t})K(b,z)}{(b_r-b_t)K(b,z)}
 =\sum_{i=1}^m\frac{b_r^2z_i/(1+b_rz_i)-b_t^2z_i/(1+b_tz_i)}{b_r-b_t}.
\]
For $r<t$ and each $i$, bringing the two fractions on the right over a common denominator gives
\begin{align}
\frac{b_r^2z_i/(1+b_rz_i)-b_t^2z_i/(1+b_tz_i)}{b_r-b_t}=\frac{z_i(b_r+b_t+b_rb_tz_i)}{(1+b_rz_i)(1+b_tz_i)}=u_{ri}+u_{ti}-u_{ri}u_{ti}.
\label{ct:kernel-pair}
\end{align}
Therefore, we have
\begin{align*}
2\sum_{1\leq r<t\leq q}\frac{(b_r^2\partial_{b_r}-b_t^2\partial_{b_t})K(b,z)}{(b_r-b_t)K(b,z)}
=2\sum_{1\leq r<t\leq q}\sum_{i=1}^m(u_{ri}+u_{ti}-u_{ri}u_{ti})=2(q-1)S-2C
\end{align*}
Each $u_{ri}$ occurs in $q-1$ pairs of alphabet indices.  
It follows from \eqref{ct:kernel-second-derivatives} and \eqref{ct:kernel-pair} that 
\begin{equation}\label{ct:kernel-b-action}
 \frac{\mathcal D_q^{(b)}K(b,z)}{K(b,z)}=-2R+2(q-1)S-2C.
\end{equation}
By interchanging the two alphabets, we obtain the separate identity
\begin{equation}\label{ct:kernel-z-action}
 \frac{\mathcal D_m^{(z)}K(b,z)}{K(b,z)}=-2C+2(m-1)S-2R.
\end{equation}
If we subtract the second identity from the first, the terms involving $R$ and $C$ cancel, leaving $2(q-m)S$. Since
\[
S=\frac{E_b K(b,z)}{K(b,z)}=\frac{E_z K(b,z)}{K(b,z)},
\]
multiplying by $K(b,z)$ establishes \eqref{ct:kernel-intertwining}. 
Because both sides are polynomials, the identity remains valid when alphabet entries are identified.
\end{proof}

\begin{lemma}\label{ct:adjoint}
Let $\mathbb F$ be a field of characteristic zero, take all
constant terms in $\mathbb F((z_1))\cdots((z_m))$, and let
$g\in\mathbb F[z_1,\ldots,z_m]$ be symmetric.  Then
\begin{equation}\label{ct:adjoint-identity}
 T_m(\mathcal D_mg)=m(m-1)T_m(g).
\end{equation}
In particular, one may take $\mathbb F=\mathbb Q(b)$, so the
identity applies when the coefficients of $g$ are polynomials
or rational functions in an auxiliary alphabet.
\end{lemma}

\begin{proof}
We use only differentiation and coefficient extraction in the
field stated in the lemma.  Write
\[
 W=\Delta_m(z)^{-2},\qquad
 \theta_i=z_i\partial_{z_i},\qquad
 A_i=\sum_{j\ne i}\frac{z_i}{z_i-z_j}.
\]
Formal differentiation in the fixed iterated Laurent-series field satisfies the product rule and agrees with rational differentiation on its rational function subfield.  
Logarithmic differentiation of $W$ yields
\begin{equation}\label{ct:weight-derivative}
 \theta_iW=-2A_iW.
\end{equation}
Indeed, the contribution of the factor with the other endpoint
$j$ is $-2z_i/(z_i-z_j)$ after division by $W$, regardless of
whether $j<i$ or $j>i$.

On every Laurent monomial one has
$z_i^2\partial_{z_i}^2=\theta_i^2-\theta_i$.  Also,
\[
 \frac{z_i^2\partial_{z_i}-z_j^2\partial_{z_j}}{z_i-z_j}
 =\frac{z_i}{z_i-z_j}\theta_i
  +\frac{z_j}{z_j-z_i}\theta_j.
\]
Consequently, when the divided differences are summed over all
unordered pairs, the coefficient of $\theta_i$ is $A_i$.
Thus \eqref{ct:operator} takes the form
\begin{equation}\label{ct:theta-operator}
\mathcal D_m=-\sum_{i=1}^m\theta_i^2+\sum_{i=1}^m\theta_i+2\sum_{i=1}^mA_i\theta_i.
\end{equation}

For every $H$ in the fixed field, $\CT_z(\theta_iH)=0$, since
$\theta_i$ multiplies the coefficient of
$z_1^{\beta_1}\cdots z_m^{\beta_m}$ by $\beta_i$.
By applying the product rule, we obtain
\[
 0=\CT_z\bigl(\theta_i(Hg)\bigr)=\CT_z\bigl((\theta_iH)g\bigr)+\CT_z\bigl(H\theta_i g\bigr).
\]
Equivalently,
\[\CT_z(H\theta_i g)=-\CT_z(g\theta_iH).
\]
Applied twice, it gives
\[
 \CT_z(W\theta_i^2g)=-\CT_z\bigl((\theta_iW)\theta_i g\bigr)=\CT_z\bigl(g\theta_i^2W\bigr).
\]
For the other two terms they yield
\[
 \CT_z(W\theta_i g)=-\CT_z(g\theta_iW),
 \qquad
 \CT_z(WA_i\theta_i g)=-\CT_z\bigl(g\theta_i(WA_i)\bigr).
\]
We therefore substitute these identities into \eqref{ct:theta-operator} to obtain
\begin{equation}\label{ct:adjoint-before-cancellation}
 T_m(\mathcal D_mg)=\CT_z g\sum_{i=1}^m\bigl(-\theta_i^2W-\theta_iW-2\theta_i(WA_i)\bigr).
\end{equation}
The second-order terms cancel.  
Equation~\eqref{ct:weight-derivative} implies
\[
 \theta_i^2W=(4A_i^2-2\theta_iA_i)W,
 \qquad
 \theta_i(WA_i)=(\theta_iA_i-2A_i^2)W.
\]
Therefore
\begin{align*}
-\theta_i^2W-\theta_iW-2\theta_i(WA_i)=\bigl(-4A_i^2+2\theta_iA_i+2A_i-2\theta_iA_i+4A_i^2\bigr)W=2A_iW.
\end{align*}
For every unordered pair $\{i,j\}$, we have 
\[
 \frac{z_i}{z_i-z_j}+\frac{z_j}{z_j-z_i}=1.
\]
It follows that $\sum_{i=1}^mA_i=\binom m2$, and substitution in \eqref{ct:adjoint-before-cancellation} yields
\[
 T_m(\mathcal D_mg)=2\binom m2\CT_z Wg=m(m-1)\,T_m(g).
\]
This proves the claim entirely within the fixed iterated Laurent series field.
\end{proof}

\begin{proposition}\label{ct:eigenfunction}
As an identity in $\mathbb Q[b_1,\ldots,b_q]$, the polynomial
$F_{m,q}$ satisfies
\begin{equation}\label{ct:eigen-equation}
 \mathcal D_qF_{m,q}=(2q-2m+1)m(m-1)\,F_{m,q}.
\end{equation}
This scalar is the diagonal eigenvalue
\begin{equation}\label{ct:eigenvalue-general}
 \varepsilon_\lambda
 =\sum_{r=1}^q
 \bigl[-\lambda_r(\lambda_r-1)+2(q-r)\lambda_r\bigr]
\end{equation}
associated with the partition $\lambda$ in
\eqref{ct:partitions}.
\end{proposition}

\begin{proof}
We may apply $T_m$ to \eqref{ct:kernel-intertwining} over $\mathbb Q(b)$. That is 
$$T_m\left(\mathcal D_q^{(b)}K(b,z)-\mathcal D_m^{(z)}K(b,z)\right)=T_m\left(2(q-m)E_bK(b,z)\right).$$
Differentiation in $b$ and multiplication by rational functions of $b$ commute with extraction of $z$ coefficients.
Thus $\mathcal D_q^{(b)}$ and $E_b$ commute with $T_m$.
This also follows directly from \eqref{ct:finite-formula}.
Therefore, we have
$$T_m(\mathcal D_q^{(b)}K(b,z))=\mathcal D_q^{(b)}T_mK(b,z)=\mathcal{D}_q^{(b)} F_{m,q}.$$
Since $K(b,z)$ is symmetric in $z$, by \cref{ct:adjoint}, we have 
$$T_m(\mathcal D_m^{(z)}K(b,z))=m(m-1)\,T_m(K(b,z))=m(m-1)F_{m,q}.$$
From \cref{ct:finite-expansion}, we get 
$$T_m(E_b K(b,z))=E_b(T_m(K(b,z)))= E_bF_{m,q}=m(m-1)F_{m,q}.$$
Consequently,
\begin{align*}
 \mathcal D_q^{(b)}F_{m,q}
 &=T_m(\mathcal D_m^{(z)}K(b,z))+2(q-m)T_m(E_bK(b,z))=m(m-1)F_{m,q}+2(q-m)E_bF_{m,q}\\
 &=(2q-2m+1)m(m-1)\,F_{m,q}.
 \end{align*}
We now compute the diagonal eigenvalue for the stated partition.
For $m=1$, $\lambda$ is zero and both eigenvalue expressions are zero.  
For $m\geq2$, in rows $2j-1$ and $2j$ of $\lambda$ the common part is $k=m-j$, where $1\leq j\leq m-1$.  
Their total contribution to \eqref{ct:eigenvalue-general} is
\[
 -2k(k-1)+\bigl[2(q-2j+1)+2(q-2j)\bigr]k=6k^2+(4q-8m+4)k.\]
All remaining parts of $\lambda$ vanish.  
From $\sum_{k=1}^{m-1}k=m(m-1)/2$ and $\sum_{k=1}^{m-1}k^2=m(m-1)(2m-1)/6$, we obtain 
\[
 \varepsilon_\lambda=6\frac{m(m-1)(2m-1)}6+(4q-8m+4)\frac{m(m-1)}2=(2q-2m+1)m(m-1),
\]
as required.
\end{proof}

\begin{theorem}\label{ct:universal-stability}
For every $m\geq1$ and $q\geq\max(1,2m-2)$, the polynomial
$F_{m,q}\in\mathbb Z[b_1,\ldots,b_q]$ is real stable.  In
other words, it is nonzero whenever all $b_r$ have positive imaginary parts.
\end{theorem}
\begin{proof}
For $m=1$, $F_{1,q}=1$.  For $m\geq2$, the partition $\lambda$ has at most $q$ parts and every adjacent difference
is zero or one.  Hence \cref{st:stable-eigenfunction} applies.
That theorem, for exactly the operator \eqref{ct:operator},
yields a unique symmetric polynomial of the form
\[
 P_\lambda=m_\lambda+\sum_{\nu\prec\lambda}p_\nu m_\nu \quad\hbox{such that}\quad
 \mathcal D_qP_\lambda=\varepsilon_\lambda P_\lambda.
\]
It also proves that $P_\lambda$ is real stable. \cref{ct:dominance} and \cref{ct:eigenfunction}
show that $F_{m,q}$ has this form and satisfies this equation.
By uniqueness, we have $F_{m,q}=P_\lambda$, and hence $F_{m,q}$ is real stable.
\end{proof}

\subsection{The two-variable specialization}

We now return to two variables.  In the coefficient field
$\mathbb Q(y,w)$, with the same ordering of the $z$ variables,
define
\begin{equation}\label{ct:homogeneous-numerator}
\mathcal H_m(y,w)=\CT_z\frac{\prod_{i=1}^m(z_i+y)^m(z_i+w)^m}{\Delta_m(z)^2}.
\end{equation}
The argument in \cref{ct:finite-expansion} shows that $\mathcal H_m(y,w)$ is a polynomial with nonnegative integer coefficients.
Every numerator monomial has total degree $2m^2$ in all of $z,y,w$.  Every monomial in the fixed
expansion of $\Delta_m(z)^{-2}$ has total $z$-degree $(-m(m-1))$.
Consequently, a numerator monomial can survive the $z$ constant
term only when its total $z$-degree is $m(m-1)$.  
Its remaining degree in $y,w$ is then $2m^2-m(m-1)=m(m+1)=L$.  This proves that
$\mathcal H_m(y,w)$ is homogeneous of degree $L$.  
By \eqref{eq:positive-ct}, we have 
\[
 \mathcal H_m(y,1)=h_{m+2}^*(y).
\]

\begin{proposition}\label{ct:specialization}
The following identity holds in $\mathbb Z[y,w]$:
\begin{equation}\label{ct:specialization-identity}
 \mathcal H_m(y,w)
 =(yw)^mF_{m,2m}(
 \underbrace{y,\ldots,y}_{m\text{ times}},
 \underbrace{w,\ldots,w}_{m\text{ times}}).
\end{equation}
Consequently, $\mathcal H_m(y,w)$ is real stable.
\end{proposition}

\begin{proof}
We first establish the identity in the rational-function field
$\mathbb Q(y,w)$.  The alphabet substitutions below are made in
the polynomial $F_{m,2m}$ given by \eqref{ct:finite-formula}. Thus, identifying its entries is legitimate.  
Here $y$ and $w$ are invertible.  For each $i$, we have 
\[
 (z_i+y)^m(z_i+w)^m=y^mw^m(1+y^{-1}z_i)^m(1+w^{-1}z_i)^m.
\]
The $m$ indices $i$ therefore contribute the common scalar $(yw)^{m^2}$.  
By pulling this scalar through $T_m$ and using \eqref{ct:universal-definition}, we obtain
\[
 \mathcal H_m(y,w)=(yw)^{m^2}F_{m,2m}(\underbrace{y^{-1},\ldots,y^{-1}}_m,\underbrace{w^{-1},\ldots,w^{-1}}_m).
\]
Now
\[
 (\underbrace{y^{-1},\ldots,y^{-1}}_{m\text{ entries}},\underbrace{w^{-1},\ldots,w^{-1}}_{m\text{ entries}})
 =\frac1{yw}(\underbrace{w,\ldots,w}_{m\text{ entries}},\underbrace{y,\ldots,y}_{m\text{ entries}}).
\]
The homogeneity of $F_{m,2m}$ of degree $m(m-1)$ thus yields
\[F_{m,2m}(\underbrace{y^{-1},\ldots,y^{-1}}_{m\text{ entries}},\underbrace{w^{-1},\ldots,w^{-1}}_{m\text{ entries}})
=(yw)^{-m(m-1)}F_{m,2m}(\underbrace{w,\ldots,w}_{m\text{ entries}},\underbrace{y,\ldots,y}_{m\text{ entries}}).
\]
By symmetry, we may interchange the two blocks of variables. 
The scalar factor then simplifies to $(yw)^{m^2-m(m-1)}=(yw)^m$, which establishes \eqref{ct:specialization-identity} in $\mathbb{Q}(y,w)$.
Since both sides belong to $\mathbb Z[y,w]$, the identity holds in that ring, including at $y=0$ or $w=0$.

If $y,w$ belong to the open upper half-plane, every entry of the specialized alphabet belongs to that half-plane. \cref{ct:universal-stability} therefore makes the $F$ factor nonzero.  
The monomial $(yw)^m$ is also nonzero there.  
This proves real stability directly from its definition.
\end{proof}

\begin{theorem}\label{ct:negative-roots}
For every $m\geq1$, the reduced polynomial $A_m(y)$ in
\eqref{eq:reduced-definition} has only negative real zeros.  Accordingly, all zeros of $h_{m+2}^*(y)$ are real and nonpositive, and its zero at the origin has multiplicity exactly $N$.
\end{theorem}
\begin{proof}
By \cref{prop:gamma-structure}, $A_m(y)$ is a polynomial of degree $2e$ with positive coefficients; in particular, its constant term is positive.
Its \emph{homogenization} to degree $2e$ is the polynomial
\[
 A_m(y,w)=\sum_{k=0}^{2e}a_{m,k}y^kw^{2e-k}.
\]
Thus $A_m(y,1)=A_m(y)$.  
Since $L=2N+2e$, the homogeneous polynomials $\mathcal H_m(y,w)$ and $y^Nw^NA_m(y,w)$ have the
same degree and the same specialization at $w=1$.  
Therefore, 
\begin{equation}\label{ct:cancel-monomial}
 \mathcal H_m(y,w)=y^Nw^NA_m(y,w).
\end{equation}
Indeed, the term $a_{m,k}y^{N+k}$ in $\mathcal H_m(y,1)$ homogenizes to
$a_{m,k}y^{N+k}w^{L-N-k}=y^Nw^Na_{m,k}y^kw^{2e-k}$.

For $y,w$ in the open upper half-plane, the monomial $y^Nw^N$ does not vanish. 
\cref{ct:specialization} and \eqref{ct:cancel-monomial} therefore imply that $A_m(y,w)$
is real stable.

Suppose that $A_m(y)=A_m(y,1)$ has a nonreal zero $\zeta$.
Its coefficients are real. Thus, complex conjugation allows us to assume $\operatorname{Im}\zeta>0$.  
Write $\zeta=u+\mathrm iv$ with $u,v\in\mathbb R$ and $v>0$. We may choose
\[
 0<\epsilon<\frac{v}{2(|u|+1)}.
\]
Then
$\operatorname{Im}(\zeta(1+\mathrm i\epsilon))=v+u\epsilon\geq v-|u|\epsilon>v/2>0$.
Thus both $w=1+\mathrm i\epsilon$ and $y=\zeta w$ lie in the open upper
half-plane, whereas homogeneity yields
\[
A_m(y,w)=A_m(\zeta w,w)=w^{2e}A_m(\zeta,1)=0.
\]
This contradicts stability.  
Hence every zero of $A_m(y)$ is real. 
Since all coefficients of $A_m(y)$ are positive, we have $A_m(y)>0$ for all $y\geq 0$, and therefore all its zeros are negative.

Finally, since $h_{m+2}^*(y)=y^N A_m(y)$ and $A_m(0)>0$, both the claimed multiplicity at the origin and the assertion concerning the remaining zeros of $h_{m+2}^*(y)$ follow immediately. When $m=1$ (equivalently, $n=3$), the explicit formulas $h_3^*(y)=y$ and $A_1(y)=1$ from \cref{prop:gamma-structure} yield the same conclusion, with the statement about the zeros of $A_1$ holding vacuously.
\end{proof}

\subsection{Proof of the conjecture}\label{sec:newton}

It remains to prove the normalized coefficient inequality in \cref{thm:resolution}.  In fact, Newton's inequalities are a corollary of the real-rootedness of polynomials with real coefficients.
A classical reference is Hardy--Littlewood--P\'olya~\cite[\S2.22, Theorem~51, p.~52;\S4.3, Theorem~144, pp.~104--105]{HLP}.
Newton's inequalities are also stated in Stanley's book \cite[Chapter 5]{RP.StanleyAC}.

\begin{lemma}{\em (Newton's inequalities, \cite[Chapter 5]{RP.StanleyAC})}\label{lem:newton}
Let $P(y)=\sum_{j=0}^d p_jy^j$ have positive real coefficients and only real roots.  For every $1\leq k<d$, we have 
\begin{equation*}
p_k^2\geq\frac{(k+1)(d-k+1)}{k(d-k)}p_{k-1}p_{k+1}.
\end{equation*}
\end{lemma}

\begin{proof}[Proof of \cref{thm:resolution}]
The formal identity in \cref{prop:positive-ct} and
the finite expansion in \cref{sec:gamma} prove
part~\textup{(i)}. \cref{prop:gamma-structure}
also proves palindromicity and strict unimodality.  In particular,
$A_m(y)\in\ZZ[y]$ has degree $2e$ and positive coefficients.

\cref{st:stable-eigenfunction} constructs the stable
eigenfunction identified with $F_{m,2m}$ in
\cref{ct:universal-stability}.
\cref{ct:specialization} and
\cref{ct:negative-roots} then prove that the roots of
$A_m(y)$ are negative.  This is part~\textup{(ii)}.
\cref{lem:newton}, applied to $P(y)=A_m(y)$ with $d=2e$,
gives \eqref{eq:main-newton} and completes part~\textup{(iii)}.
When $e=0$, we have $A_1=1$ and there is no inequality to check.
This completes the proof.
\end{proof}

\section{By-products}\label{sec:consequences}

Define
\begin{equation}\label{eq:gamma-polynomial}
\Gamma_m(s)=\sum_{j=0}^e\gamma_{m,j}s^j.
\end{equation}
Its coefficients are positive and its degree is $e$.  The flow expansion also yields
\begin{equation}\label{eq:gamma-ct}
s^N\Gamma_m(s)=\CT_z\frac{\prod_{i=1}^m(z_i^2+z_i+s)^m}{\Delta_m(z)^2}.
\end{equation}

To see this, we may replace $(1+y)z_i$ by $z_i$ and $y$ by $s$ in \eqref{eq:trinomial}. 
This leaves the balance equations and weights unchanged.
The variable $s$ has exponent $\mathsf C$.
By collecting terms according to $\mathsf C=N+j$, we obtain \eqref{eq:gamma-ct}, including divisibility by $s^N$ in $\ZZ[s]$.
The constant term is taken in $\mathbb Q(s)((z_1))\cdots((z_m))$.

\begin{proposition}\label{prop:root-equivalence}
The following three statements are equivalent for every $m\geq1$:
\begin{enumerate}
\item[(i)] $h_{m+2}^*(y)$ is real-rooted;
\item[(ii)] $A_m(y)$ is real-rooted;
\item[(iii)] every root of $\Gamma_m(s)$ is a negative real number.
\end{enumerate}
\end{proposition}
\begin{proof}
The first equivalence follows from $h_{m+2}^*(y)=y^NA_m(y)$ and $A_m(0)>0$.  
By the gamma expansion for $y\neq-1$, we obtain 
\begin{equation}\label{eq:gamma-transform}
A_m(y)=(1+y)^{2e}\Gamma_m\left(\frac{y}{(1+y)^2}\right).
\end{equation}
At $y=-1$, only the last gamma term survives:
\begin{equation}\label{eq:minus-one}
A_m(-1)=(-1)^e\gamma_{m,e}\neq0.
\end{equation}
For $e=0$, we have $h_3^*(y)=y$ and $A_1=\Gamma_1=1$. Thus, all three statements hold.  
We now set $e\geq1$.

Suppose first that $\Gamma_m(s)=\gamma_{m,e}\prod_{j=1}^e(s-\rho_j)$ with $\rho_j<0$.  
Equation \eqref{eq:gamma-transform} becomes
\[
A_m(y)=\gamma_{m,e}\prod_{j=1}^e\bigl(y-\rho_j(1+y)^2\bigr).
\]
Writing $c_j=-\rho_j>0$, the $j$-th factor is $c_jy^2+(1+2c_j)y+c_j$.  
Its discriminant is $(1+2c_j)^2-4c_j^2=1+4c_j>0$.
Its two roots have positive product one and negative sum $-(1+2c_j)/c_j$. Thus, both are negative real numbers.
This proves \textup{(ii)}.

Conversely, suppose $A_m$ is real-rooted.  Its positive coefficients force every root to be negative.  
If $\rho$ is a zero of $\Gamma_m$, then $\rho\neq0$ because $\Gamma_m(0)>0$.  
Hence $y=\rho(1+y)^2$ is a quadratic equation with a complex solution $y_0$; substitution shows that $y_0\neq-1$.  
By \eqref{eq:gamma-transform}, we have $A_m(y_0)=0$. Thus, $y_0<0$ and $\rho=y_0/(1+y_0)^2<0$.  This proves \textup{(iii)}.
\end{proof}

\begin{corollary}\label{cor:gamma-roots}
Every root of $\Gamma_m$ is negative.  If $e\geq2$, its coefficients satisfy
\[
\gamma_{m,j}^2\geq
\frac{(j+1)(e-j+1)}{j(e-j)}
\gamma_{m,j-1}\gamma_{m,j+1}\qquad(1\leq j<e).
\]
\end{corollary}
\begin{proof}
It follows from \cref{thm:resolution} and \cref{prop:root-equivalence} that every root of $\Gamma_m$ is negative.
The coefficient inequality follows from \cref{lem:newton} applied to $\Gamma_m$, whose coefficients are positive by \cref{prop:gamma-structure}.
\end{proof}

We now give some small cases of $\Gamma_m(s)$ as follows:
\begin{align*}
\Gamma_1(s)&=1,\\
\Gamma_2(s)&=1+2s,\\
\Gamma_3(s)&=9+36s+12s^2,\\
\Gamma_4(s)&=36+960s+3924s^2+3312s^3+312s^4.
\end{align*}
These values follow by finite enumeration in \eqref{eq:gamma-flow-formula}. 
For $m=2$, the last vertex forces $f_{12}=0$ and $(a_2,b_2,c_2)=(0,0,2)$.
At the first vertex $2a_1+b_1=2$, so the possibilities are $(a_1,b_1,c_1)=(0,2,0)$ and $(1,0,1)$, with respective weights one and two. This gives $\Gamma_2=1+2s$.






\noindent
{\small \textbf{Acknowledgments:}}
The authors would like to express sincere gratitude for all the suggestions that have improved the presentation of this paper.
Feihu Liu was partially supported by the Postdoctoral Fellowship Program and China Postdoctoral Science Foundation (Grant No. BX2026002).

\noindent{\small \textbf{Declaration of AI Assistance:}}

During the preparation of this manuscript, the authors used ChatGPT to assist in exploring possible approaches, checking technical details, and improving the exposition. All mathematical arguments, computations, proofs, and results were independently verified by the authors. The authors take full responsibility for the content of this manuscript.


\begin{thebibliography}{99}
    
\bibitem{MorrisBV} W. Baldoni-Silva and M. Vergne, Residues formulae for volumes and Ehrhart polynomials of convex polytopes.	arXiv:math/0103097, (2001).
    
\bibitem{BeckPixton} M. Beck and D. Pixton, The Ehrhart polynomial of the Birkhoff polytope, \emph{Discrete Comput. Geom.} 30 (2003), no.~4, 623--637.

\bibitem{forrester2008} P. J. Forrester and S. O. Warnaar, The importance of the Selberg integral, \emph{Bull. Amer. Math. Soc}, 45(4) (2008), 489--534.

\bibitem{HLP} G.~H. Hardy, J.~E. Littlewood, and G. P\'olya, \emph{Inequalities}, 2nd ed., Cambridge University Press, Cambridge, (1952).

\bibitem{morales2021} A. H. Morales and W. Shi, Refinements and Symmetries of the Morris identity for volumes of flow polytopes, \emph{Comptes Rendus. Math\'ematique}, 359 (2021), 823--851.

\bibitem{MorrisThesis} W. G. Morris, \emph{Constant term identities for finite and affine root systems}. Ph.D. thesis, University of Wisconsin-Madison. (1982).

\bibitem{selberg1944} A. Selberg, Bemerkninger om et multipelt integral, \emph{Nordisk Mat. Tidskr.} 26 (1944), 71--78.


\bibitem{StanleyJack} R. P. Stanley, Some combinatorial properties of Jack symmetric functions, \emph{Adv. Math.} 77 (1989), no.~1, 76--115.

\bibitem{RP.StanleyAC} R. P. Stanley, \emph{Algebraic Combinatorics: Walks, Trees, Tableaux, and More}, second edition, Undergraduate Texts in Mathematics. Springer, (2018).

\bibitem{Trotter} H. F. Trotter, On the product of semi-groups of operators, \emph{Proc. Amer. Math. Soc.} 10 (1959), no.~4, 545--551.

\bibitem{XinThesis} G. Xin, \emph{The Ring of Malcev--Neumann Series and the Residue Theorem}, Ph.D. thesis, Brandeis University, arXiv:math/0405133v1. (2004).

\bibitem{XinFast} G. Xin, A fast algorithm for MacMahon's partition analysis, \emph{Electron. J. Combin.} 11 (2004), \#R58.

\bibitem{XinEuclid} G. Xin, A Euclid style algorithm for MacMahon's partition analysis, \emph{J. Combin Theory, Series A} 131 (2015), 32--60.

\bibitem{XZ} G. Xin and C. Zhang, A variation of the Morris constant term, Electron. J. Combin. 33(3), \#P3.75.

\bibitem{xin2022} G. Xin, C. Zhang, Y. Zhou and Y. Zhong, The constant term algebra of type A: the structure, \emph{Adv. Math.} 465 (2025), 110154.

\bibitem{Zeilberger} D. Zeilberger, Proof of a conjecture of Chan, Robbins, and Yuen. \emph{Electron. Trans. Numer. Anal.} 9 (1999), 147--148.

\bibitem{ZhouYue} Y. Zhou, The AFLT $q$-Morris constant term identity. \emph{Adv. in Appl. Math.} 147 (2023), No. 102506.

\bibitem{ZhouLuFu} Y. Zhou, J. Lu, and H. Fu, Leading coefficients of Morris type constant term identities. \emph{Adv. in Appl. Math.} 87 (2017), 24--42.

\end{thebibliography}
\end{document}